\documentclass[a4paper,11pt]{article}
\usepackage[latin1]{inputenc}
\usepackage[english]{babel}
\usepackage{amsmath}
\usepackage{amsfonts}
\usepackage{amssymb}
\usepackage{epsfig}
\usepackage{amsopn}
\usepackage{amsthm}
\usepackage{color}
\usepackage{graphicx}
\usepackage{subfigure}
\usepackage{enumerate}
\usepackage{subcaption}
\newtheorem{theorem}{Theorem}[section]
\newtheorem{corollary}[theorem]{Corollary}
\newtheorem{lemma}[theorem]{Lemma}
\newtheorem{proposition}[theorem]{Proposition}
\newtheorem{definition}[theorem]{Definition}

\newtheorem*{theorem*}{Theorem}
\newtheorem*{lemma*}{Lemma}
\newtheorem*{remark*}{Remark}
\newtheorem*{definition*}{Definition}
\newtheorem*{proposition*}{Proposition}
\newtheorem*{corollary*}{Corollary}
\numberwithin{equation}{section}
\newcommand{\real}{\mathbb{R}}

\def\qed{\,\unskip\kern 6pt \penalty 500
\raise -2pt\hbox{\vrule \vbox to8pt{\hrule width 6pt
\vfill\hrule}\vrule}\par}
\definecolor{darkblue}{rgb}{0.05, .05, .65}
\definecolor{darkgreen}{rgb}{0.1, .65, .1}
\definecolor{darkred}{rgb}{0.8,0,0}
\newcommand{\beqn}{\begin{equation}}
\newcommand{\eeqn}{\end{equation}}
\newcommand{\bear}{\begin{eqnarray}}
\newcommand{\eear}{\end{eqnarray}}
\newcommand{\bean}{\begin{eqnarray*}}
\newcommand{\eean}{\end{eqnarray*}}
\begin{document}

\title{\huge \bf Large time behavior and transition from vanishing to spreading regimes for the generalized Burgers-Fisher-KPP equation}

\author{
\Large Razvan Gabriel Iagar\,\footnote{Departamento de Matem\'{a}tica
Aplicada, Ciencia e Ingenieria de los Materiales y Tecnologia
Electr\'onica, Universidad Rey Juan Carlos, M\'{o}stoles,
28933, Madrid, Spain, \textit{e-mail:} razvan.iagar@urjc.es},\\
[4pt] \Large Ariel S\'{a}nchez,\footnote{Departamento de Matem\'{a}tica
Aplicada, Ciencia e Ingenieria de los Materiales y Tecnologia
Electr\'onica, Universidad Rey Juan Carlos, M\'{o}stoles,
28933, Madrid, Spain, \textit{e-mail:} ariel.sanchez@urjc.es}\\
[4pt] }
\date{}
\maketitle

\begin{abstract}
The large time behavior of solutions to the following generalized Burgers-Fisher-KPP equation
$$
\partial_tu=u_{xx}+k(u^n)_x+u^p-u^q, \quad (x,t)\in\real\times(0,\infty),
$$
with $n\geq2$, $p>q\geq1$ and $k\in\real$, is considered in this work. Denoting by $H(x,t)$, respectively $\widetilde{H}(x,t)$ the solutions having as initial condition the Heaviside, respectively the ``anti-Heaviside" functions
$$
H_0(x)=\begin{cases}
  0, & \mbox{if } x<0 \\
  1, & \mbox{if } x\geq0.
\end{cases}, \quad \widetilde{H}_0(x)=1-H_0(x),
$$
critical velocities $\overline{c}$, respectively $\widetilde{c}=kn+2\sqrt{p-q}$, are identified such that $H(x,t)$, respectively $\widetilde{H}(x,t)$ approach the unique traveling wave solution of the equation with these critical velocities as $t\to\infty$. The critical velocity $\overline{c}$ is \emph{anomalous}, that is, it cannot be made explicit by an algebraic expression. Assuming for simplicity $k>0$, a remarkable fact is that, while $\widetilde{H}(x,t)\to0$ as $t\to\infty$ uniformly on compact subsets of $\real$, the Heaviside solution $H$ might tend either to zero or to one as $t\to\infty$, depending on the sign of the critical velocity $\overline{c}$. This sign vary with respect to the exponents $n$, $p$, $q$ and the coefficient $k$ and, in fact, we prove that given $p$, $q$, $n$, there exists a critical coefficient $k^*(n,p,q)$ such that $\overline{c}>0$ if $k>k^*(n,p,q)$ and $\overline{c}<0$ if $k<k^*(n,p,q)$. The convergence to either zero or one reflects the sharp influence of the convection term, since in the absence of it (that is, $k=0$), $H(x,t)$ would always tend to zero as $t\to\infty$. The results include more general initial conditions than the Heaviside-type functions, and sharp estimates of the threshold coefficient $k^*(n,p,q)$ are also given.
\end{abstract}

\

\noindent {\bf Mathematics Subject Classification 2020:} 35A18, 35B33, 35B40, 35C07, 35K58, 35Q51.

\smallskip

\noindent {\bf Keywords and phrases:} generalized Burgers-Fisher-KPP equation, Heaviside function, large time behavior, convergence to traveling waves, critical speed.

\section{Introduction}

The Fisher-KPP equation, written in a more general form as
\begin{equation}\label{FKPP}
\partial_tu=u_{xx}+u^p-u^q,
\end{equation}
has been introduced in the papers by Fisher \cite{Fi37} (for $p=1$ and $q=2$) and Kolmogorov, Petrovsky and Piscounoff \cite{KPP37} (for a more general reaction $f(u)$ having as model $u-u^q$, $q>1$) following from modeling in population dynamics and became in recent times one of the most established equations in applied mathematics. The competition between the diffusion, reaction and absorption terms is well understood nowadays whenever $p<q$, that is, in the classical range of exponents investigated originally by the authors of \cite{KPP37}. Many extensions to quasilinear diffusion, or to more general reaction-absorption terms written in the form of a general function $f(u)$ with suitable properties have been considered, including also extensions to the $N$-dimensional case. Thus, a wide amount of literature establishing qualitative properties and the large time behavior of the solutions to equations that are a variation of \eqref{FKPP} is nowadays available, see for example \cite{AW75, Biro02, B83, DGQ20, DQZ20, G20, GK, Sat76, SGM94}.

One of the main features of Eq. \eqref{FKPP} is the convergence of general solutions of it towards a special class of solutions known as \emph{traveling waves}. These solutions have the particular form
\begin{equation}\label{TWS}
u(x,t)=f(x+ct), \quad c\in\real, \quad (x,t)\in\real\times(0,\infty),
\end{equation}
where $c$ is the velocity and $f$ the profile of the wave. A celebrated result (see \cite{Fi37, KPP37}) is the fact that there exists a minimal velocity $c^*=2$ such that traveling wave solutions in the form \eqref{TWS} to Eq. \eqref{FKPP} with $0\leq u\leq1$ exist for $|c|\geq c^*$ and do not exist for $|c|<c^*$. A further important property of the traveling wave solutions is that they are asymptotic limits as $t\to\infty$ of general solutions stemming from initial conditions $u_0$ satisfying $0\leq u_0(x)\leq1$ for any $x\in\real$ and $u_0(x)\to0$ as $x\to-\infty$, $u_0(x)\to1$ as $x\to\infty$ or viceversa. To quote but a few papers where such large time behavior is established for Eq. \eqref{FKPP} or more general versions of it, we refer the reader to \cite{KPP37, Biro02, U78, DGQ20, G20} and references therein. In particular, denoting by
\begin{equation}\label{Heaviside}
H_0(x)=\begin{cases}
       0, & \mbox{if } x<0 \\
       1, & \mbox{if } x\geq0,
     \end{cases}
\end{equation}
the Heaviside function and considering the solution $H(x,t)$ to the Cauchy problem associated to Eq. \eqref{FKPP} with $H_0(x)$ as initial condition, it has been proved in \cite{KPP37} that it spreads (that is, tends to the constant solution equal to 1) with exactly the minimal velocity $c^*=2$ and approaches as $t\to\infty$ the (unique, modulo translation) traveling wave with velocity $c^*=2$.

The effect of a convection term inserted into a Fisher-KPP equation has been studied in more recent works. The existence of a minimal velocity for the traveling wave solutions is established in \cite{GK05, MM02, MM03, MaOu21, PPS25}, and very recent papers provide a rather complete study of the large time behavior of more general solutions towards traveling wave solutions, see \cite{Xu24, AHR25} and references therein. In particular, \cite{Xu24} adapts the technique in \cite{Biro02} to include a convection term and prove convergence of a rather general class of solutions to the Burgers-Fisher-KPP equation with degenerate diffusion towards the traveling wave solution with minimal velocity. 

In this paper, we consider the generalized Burgers-Fisher-KPP equation
\begin{equation}\label{eq1}
u_t=u_{xx}+k(u^n)_x+u^p-u^q, \quad (x,t)\in\real\times(0,\infty),
\end{equation}
in the range of exponents and coefficient
\begin{equation}\label{range.exp}
n\geq2, \quad p>q\geq1, \quad k>0.
\end{equation}
In fact, our results cover Eq. \eqref{eq1} with any $k\in\real$. Indeed, considering negative convection coefficients $k<0$ reduces to Eq. \eqref{eq1} after performing the simple transformation $x=-y\in\real$ and the reader can easily translate our results through this transformation. Thus, we will assume for simplicity that $k>0$ in the proofs.

Actually, Eq. \eqref{eq1} is a model that can be straightforwardly generalized, since our study can be readily extended to the more general equation
\begin{equation}\label{eq1.gen}
\partial_tv=Av_{xx}+B(v^n)_x+Cv^p-Dv^q,
\end{equation}
posed for $(x,t)\in\real\times(0,\infty)$, with arbitrary positive coefficients $A$, $C$, $D>0$ and arbitrary $B\in\real$. Indeed, the following rescaling
$$
v(x,t)=\lambda u(\mu x,\nu t), \quad \lambda=\left(\frac{D}{C}\right)^{1/(p-q)}, \quad \nu=C^{(1-q)/(p-q)}D^{(p-1)/(p-q)}, \quad \mu=\sqrt{\frac{\nu}{A}}
$$
maps Eq. \eqref{eq1.gen} into Eq. \eqref{eq1}, taking
$$
k=\frac{B\lambda^{n-1}\mu}{\nu}=\frac{B}{\sqrt{A}}C^{(q+1-2n)/2(p-q)}D^{-(p+1-2n)/2(p-q)}.
$$

We observe that Eq. \eqref{eq1} features the opposite order between the exponents of the reaction and absorption terms with respect to the
classical Fisher-KPP equation. The range $p>q$ can be physically interpreted as a process in which the growth mechanism dominates the absorption at high densities. This can be for example the case in combustion processes where the heat production dominates the dissipative effects or in chemical reactions in which the production rate accelerates with the increase in concentration (a phenomenon occurring in chain reactions). A more general and detailed motivation for general equations including Eq. \eqref{eq1} and their traveling waves can be found in the monograph \cite{VVV}.

The non-convective counterpart of Eq. \eqref{eq1} with $k=0$ arises (together with close relatives) in models mixing growth and diffusion, see for example \cite{Ba94, Ma10}, and traveling wave solutions have been analyzed in \cite{SHB05, HBS14}. Moreover, in papers such as \cite{HBS14, IS25}, it is shown that either the traveling wave solutions (in the range $p>1>q$) or a stationary solution (in the range $p>q\geq1$) play the role of threshold solutions separating between the vanishing regime (satisfied by general solutions lying below them) and the blow-up regime (satisfied by general solutions lying above them). A number of recent papers \cite{LH24, PS25, Zhang21, Zhang22} (see also the references therein) study the existence of isolated periodic traveling waves, having as common point a technique based on an abelian integral related to an alternative formulation of Eq. \eqref{eq1} as a dynamical system seen as a small perturbation of a Hamiltonian system.

As we shall analyze in this paper, one remarkable effect of the presence of a convection term is that at the same time some solutions whose initial conditions satisfy $u_0(x)\to0$ as $x\to-\infty$ and $u_0(x)\to1$ as $x\to\infty$ might converge to the constant solution $u\equiv1$ locally uniformly (a phenomenon usually referred in literature as \emph{spreading}) and other solutions stemming from initial conditions with similar properties might converge to the constant solution $u\equiv0$ (a phenomenon known as \emph{vanishing}). Thus, the convection creates a situation in which none of the two possible outcomes is universally dominant. Instead, Eq. \eqref{eq1} exhibits a selection mechanism: depending on specific features of the initial condition, the solution may converge to one of the two equilibria.

In their previous paper \cite{IS26}, the authors classified the traveling wave solutions in the form \eqref{TWS} to Eq. \eqref{eq1} in the range of exponents \eqref{range.exp}. This classification is performed on an autonomous dynamical system, and represents the starting point of the present work. We thus apply in this paper the properties of the traveling waves deduced in \cite{IS26} in order to establish the behavior as $t\to\infty$ of suitable classes of general solutions, including the Heaviside and ``anti-Heaviside" solutions, that are maximal in the sense of the velocity of spreading/vanishing. We show that the Heaviside solution might either vanish or spread according to the sign of a \emph{critical velocity} that is \emph{anomalous} (in the sense that its value is not obtained as an algebraic expression of the parameters of the problem, but it is deduced from a dynamical system, according to the terminology in \cite{VazSmooth}). Moreover, the transition from a vanishing to a spreading regime for this solution, given the exponents $n$, $p$, $q$, only depends on the coefficient $k$ of the convective term in Eq. \eqref{eq1}, a fact that is a very unexpected effect of the convection.

Before introducing the main results, we recall in the next section some preliminary facts proved in \cite{IS26} about the traveling wave solutions to Eq. \eqref{eq1}.

\section{Preliminary facts}

We recall in this section some specific results related to the classification of traveling waves to Eq. \eqref{eq1}, following \cite{IS26}. In order to study them, an alternative formulation of Eq. \eqref{eq1} in the form of an autonomous dynamical system has been employed. Inserting the ansatz \eqref{TWS} into Eq. \eqref{eq1}, we deduce that the profiles $f$ of traveling waves solve the differential equation
\begin{equation}\label{TWODE}
f''(\xi)=cf'(\xi)-knf^{n-1}(\xi)f'(\xi)-f^p(\xi)+f^q(\xi).
\end{equation}
By setting
\begin{equation}\label{PSchange}
X(\xi)=f(\xi)>0, \quad Y(\xi)=f'(\xi)\in\real, \quad \xi\in\real,
\end{equation}
Eq. \eqref{TWODE} is transformed into the following autonomous first order system
\begin{equation}\label{PSsyst}
\left\{\begin{array}{ll}X'=Y, \\Y'=cY-knX^{n-1}Y-X^p+X^q.\end{array}\right.
\end{equation}
The system \eqref{PSsyst} has two critical points $P_1=(0,0)$ and $P_2=(1,0)$. We recall the local behavior of the system \eqref{PSsyst} in the following statements.
\begin{proposition}[Propositions 2.1 in \cite{IS26}]\label{prop.P1}
Let $n$, $p$, $q$, $k$ as in \eqref{range.exp} such that $q>1$ and $c\neq0$. Then the critical point $P_1$ is non-hyperbolic and it has a unique trajectory going out of it into the cone $X>0, Y>0$ of the phase plane associated to the system \eqref{PSsyst} and a unique trajectory connecting to it from the region $X>0, Y<0$ of the phase plane. The local behavior of the profiles corresponding to these trajectory is given by:

(a) If $c>0$, then the unique trajectory contained in the unstable manifold of $P_1$ contains a one-parameter family of profiles such that
\begin{equation*}
\lim\limits_{\xi\to-\infty}e^{-c\xi}f(\xi)=L\in(0,\infty),
\end{equation*}
while the unique trajectory contained in the center manifold corresponds to profiles such that
\begin{equation*}
\lim\limits_{\xi\to\infty}\xi^{1/(q-1)}f(\xi)=\left(\frac{c}{q-1}\right)^{1/(q-1)}.
\end{equation*}

(b) If $c<0$, then the unique trajectory contained in the stable manifold of $P_1$ contains a one-parameter family of profiles satisfying
\begin{equation*}
\lim\limits_{\xi\to\infty}e^{-c\xi}f(\xi)=L\in(0,\infty),
\end{equation*}
while the unique trajectory contained in the center manifold corresponds to profiles satisfying
\begin{equation*}
\lim\limits_{\xi\to-\infty}|\xi|^{1/(q-1)}f(\xi)=\left(\frac{|c|}{q-1}\right)^{1/(q-1)}.
\end{equation*}
\end{proposition}
If $c=0$, the local behavior is rather different and it depends on some particular relations between $n$, $p$ and $q$. Since the case $c=0$ will be very important as borderline case, we give below the local behavior of the profiles connecting to $P_1$ with $c=0$.
\begin{proposition}[Proposition 2.2 in \cite{IS26}]\label{prop.P1c0}
Let $n$, $p$, $q$, $k$ as in \eqref{range.exp}. When $c=0$, both the stable and unstable trajectories of $P_1$ are tangent to the $X$-axis in the phase plane. The traveling wave profiles corresponding to these trajectories have the following local behavior:
\begin{enumerate}
  \item If $n>(q+1)/2$, then the unstable trajectory, respectively the stable trajectory, correspond to
\begin{equation*}
\begin{split}
&\lim\limits_{\xi\to-\infty}|\xi|^{2/(q-1)}f(\xi)=\left[\frac{q-1}{2}\sqrt{\frac{2}{q+1}}\right]^{-2/(q-1)}, \quad {\rm respectively}, \\
&\lim\limits_{\xi\to\infty}\xi^{2/(q-1)}f(\xi)=\left[\frac{q-1}{2}\sqrt{\frac{2}{q+1}}\right]^{-2/(q-1)}.
\end{split}
\end{equation*}
  \item If $n=(q+1)/2$, let us introduce
\begin{equation*}
v_1:=\frac{\sqrt{k^2n^2+4n}-kn}{2n}, \quad v_2:=\frac{-\sqrt{k^2n^2+4n}-kn}{2n}.
\end{equation*}
Then the unstable trajectory, respectively the stable trajectory, correspond to
\begin{equation*}
\begin{split}
&\lim\limits_{\xi\to-\infty}|\xi|^{1/(n-1)}f(\xi)=[v_1(n-1)]^{-1/(n-1)}, \quad {\rm respectively},\\
&\lim\limits_{\xi\to\infty}\xi^{1/(n-1)}f(\xi)=[-v_2(n-1)]^{-1/(n-1)}.
\end{split}
\end{equation*}
  \item If $2\leq n<(q+1)/2$, then the unstable trajectory, respectively the stable trajectory, correspond to
\begin{equation*}
\begin{split}
&\lim\limits_{\xi\to-\infty}|\xi|^{1/(q-n)}f(\xi)=\left[\frac{q-n}{kn}\right]^{-1/(q-n)}, \quad {\rm respectively},\\
&\lim\limits_{\xi\to\infty}\xi^{1/(n-1)}f(\xi)=\left[k(n-1)\right]^{-1/(n-1)}.
\end{split}
\end{equation*}
\end{enumerate}
\end{proposition}
Finally, in the limiting case $q=1$, and for any $c\in\real$, we introduce
$$
\lambda_1:=\frac{c+\sqrt{c^2+4}}{2}>0, \quad \lambda_2:=\frac{c-\sqrt{c^2+4}}{2}<0.
$$
Then we have
\begin{proposition}[Proposition 2.3 in \cite{IS26}]\label{prop.P1q1}
Let $n$, $p$, $q$, $k$ as in \eqref{range.exp}. For $q=1$, the critical point $P_1$ is a saddle point. The unique trajectory contained in its unstable manifold, respectively its stable manifold, corresponds to profiles such that:
\begin{equation}\label{beh.P1q1}
\lim\limits_{\xi\to-\infty}e^{-\lambda_1\xi}f(\xi)\in(0,\infty), \quad {\rm respectively} \quad
\lim\limits_{\xi\to\infty}e^{-\lambda_2\xi}f(\xi)\in(0,\infty).
\end{equation}
\end{proposition}
The uniqueness of the trajectory of the system \eqref{PSsyst} stemming from $P_1$ on its unstable or center manifold is a very important feature of Eq. \eqref{TWODE}. We denote in the sequel by $l_c$ this trajectory and by $f_c$ the corresponding profile of a traveling wave obtained by undoing the transformation \eqref{PSchange} for a given $c\in\real$.

The critical point $P_2$ changes its behavior according to the values of $c$. More precisely,
\begin{proposition}[Proposition 2.4 in \cite{IS26}]\label{prop.P2}
Let $n$, $p$, $q$, $k$ as in \eqref{range.exp}. The critical point $P_2$ is
\begin{itemize}
  \item An unstable node if $c\geq kn+2\sqrt{p-q}$.
  \item An unstable focus if $c\in(kn,kn+2\sqrt{p-q})$.
  \item A stable focus if $c\in(kn-2\sqrt{p-q},kn)$.
  \item A stable node if $c\leq kn-2\sqrt{p-q}$.
\end{itemize}
Moreover, if $n\neq p+q+1$, a Hopf bifurcation occurs at $c=kn$, generating at least one limit cycle in the system \eqref{PSsyst} either for $c<kn$ when $n<p+q+1$, or for $c>kn$ when $n>p+q+1$.
\end{proposition}
A pair of numbers related to the critical point $P_2$ that are very useful throughout the paper are the eigenvalues and the corresponding eigenvectors of its linearized matrix, namely
\begin{equation}\label{eigen.P2}
\begin{split}
&\lambda_+(c):=\frac{c-kn+\sqrt{(c-kn)^2-4(p-q)}}{2}, \quad e_+(c)=(1,\lambda_+(c)) \\ 
&\lambda_-(c):=\frac{c-kn-\sqrt{(c-kn)^2-4(p-q)}}{2}, \quad e_{-}(c)=(1,\lambda_{-}(c)).
\end{split}
\end{equation} 
The analysis performed in \cite{IS26} shows that, in striking contrast to the standard Fisher-KPP equation, Eq. \eqref{eq1} admits traveling waves with any velocity $c\in\real$. This fact makes it much more involved to determine a specific velocity of spreading or vanishing for a specific initial condition (or for a class of them), since apparently there is no limit value $c^*$ as for Eq. \eqref{FKPP}. A complete classification of the traveling waves with respect to the value of $c\in(-\infty,\infty)$ according to their behavior at both ends has been performed in \cite[Theorems 1.1 and 1.2]{IS26}. In the present paper, we will only employ some specific trajectories of the system \eqref{PSsyst} connecting the critical points $P_1$ and $P_2$, thus we refrain from recalling the full and rather lengthy classification. We are now in a position to state our main results.

\section{Main results}

Before stating our main results, let us introduce some notation that will be employed throughout the paper. One of the main goals of this paper is to describe the large time behavior of the solutions obtained from the Heaviside and the anti-Heaviside functions as initial conditions. Let thus $H(x,t)$ and $\widetilde{H}(x,t)$ be the solutions to Eq. \eqref{eq1} such that
\begin{equation*}
H(x,0)=H_0(x), \quad \widetilde{H}(x,0)=1-H_0(x),
\end{equation*}
where $H_0$ is the usual Heaviside function defined in \eqref{Heaviside}. Let us next recall from Proposition \ref{prop.P2} that the critical point $P_2$ in the system \eqref{PSsyst} becomes a stable node for $-\infty<c\leq kn-2\sqrt{p-q}$.
\begin{definition}\label{def.con}
We say that the trajectory $l_c$ \emph{connects directly} to $P_2$ if it ends in the stable node $P_2$ without crossing the $X$-axis; that is,
$$
X(\xi)>0, \quad Y(\xi)>0, \quad \xi\in(-\infty,\infty)
$$
on this trajectory.
\end{definition}
It has been proved in \cite[Proposition 6.1]{IS26} that the trajectory $l_c$ connects directly $P_1$ and $P_2$ at least for any $c\leq-2\sqrt{p-q}$. Let us thus define the \emph{critical velocity}
\begin{equation}\label{crit.vel}
\overline{c}:=\sup\{c\in\real: l_c \ {\rm connects \ directly} \ P_1 \ {\rm and} \ P_2\}.
\end{equation}
As we shall see in Lemma \ref{lem.crit}, $\overline{c}$ is the biggest velocity for which there is a trajectory connecting directly $P_1$ to $P_2$ in the sense of Definition \ref{def.con} and thus, the biggest velocity for which there is a traveling wave solution in the form \eqref{TWS} with $0\leq u(x,t)\leq1$, $(x,t)\in\real\times(0,\infty)$. In particular, the trajectory $l_{\overline{c}}$ will enter $P_2$ tangent to the eigenvector $e_{-}(\overline{c})$ defined in \eqref{eigen.P2}. Moreover, $\overline{c}$ can take any of the two signs (positive or negative). Let us also remark at this point that, when $\overline{c}>0$, there are at the same time traveling wave solutions $u$ in the form \eqref{TWS} with $0\leq u\leq1$ such that $u(x,t)\to1$ locally uniformly as $t\to\infty$ (with velocity $c\in(0,\overline{c})$) and such that $u(x,t)\to0$ locally uniformly as $t\to\infty$ (with velocity $c<0$). When $\overline{c}<0$, all the traveling wave solutions $u$ in the form \eqref{TWS} such that $0\leq u\leq1$ converge to zero as $t\to\infty$.

In this notation, we can state a rather striking convergence result for the Heaviside solution.
\begin{theorem}\label{th.Heav}
Let $k$, $n$, $p$, $q$ be as in \eqref{range.exp}. Then the Heaviside solution $H(x,t)$ evolves with the velocity $\overline{c}$. More precisely, we have (with locally uniform convergence)
\begin{equation}\label{convH1}
\lim\limits_{t\to\infty}H(x-ct,t)=\begin{cases}
                                    1, & \mbox{if } c<\overline{c} \\
                                    0, & \mbox{if } c>\overline{c},
                                  \end{cases}
\end{equation}
and the following convergence in form to the traveling wave with profile $f_{\overline{c}}$
\begin{equation}\label{convH2}
\lim\limits_{t\to\infty}\sup_{x\in\real}|H(x,t)-f_{\overline{c}}(x+\overline{c}t)|=0.
\end{equation}
\end{theorem}
Let us first observe that $\overline{c}\in[-2\sqrt{p-q},kn-2\sqrt{p-q}]$. Indeed, the lower bound is ensured by \cite[Proposition 6.1]{IS26}, while the upper bound follows from the fact that, for $c>kn-2\sqrt{p-q}$, the critical point $P_2$ becomes a stable focus and thus a direct connection to it (in the sense of Definition \ref{def.con}) is no longer possible. With respect to the solution $H$, it can either vanish or spread, according to Theorem \ref{th.Heav} and the sign of the critical velocity $\overline{c}$. This fact is a striking effect of the convection term, since it is obvious that, in the absence of the convective term (that is, for $k=0$), we have $\overline{c}=-2\sqrt{p-q}<0$ and $H(x,t)\to0$ locally uniformly as $t\to\infty$. The outcome of Theorem \ref{th.Heav} can be interpreted as a balance between the competing terms in Eq. \eqref{eq1}. Indeed, while for $\overline{c}<0$ we recognize the same qualitative outcome as in the non-convective case $k=0$, when $\overline{c}>0$ the convection term and the reaction one reinforce each other, and the mass is transported sufficiently fast to determine stabilization to the positive state $u\equiv1$.

Let us mention here that we have kept as initial condition the Heaviside function $H_0(x)$ in Theorem \ref{th.Heav} only for the simplicity of the statement. Actually, a rather general class of initial conditions $u_0$ produces solutions to the Cauchy problem associated to Eq. \eqref{eq1} for which Theorem \ref{th.Heav} remains true. Such a generalization is given in Theorem \ref{th.Heav.gen}, stated and proved at the end of Section \ref{sec.Heav}, together with a discussion of other possible cases.

Before continuing the discussion related to the critical velocity $\overline{c}$, let us discuss first the large time behavior of the ``anti-Heaviside" solution $\widetilde{H}$, whose initial condition is $1-H_0(x)$. In order to ease the notation, we denote by $\widetilde{c}:=kn+2\sqrt{p-q}$ the exact velocity for which the critical point $P_2$ changes from an unstable node to an unstable focus. The solution $\widetilde{H}$ vanishes as $t\to\infty$ with the velocity of propagation $\widetilde{c}$, as stated below.
\begin{theorem}\label{th.antiHeav}
Let $k$, $n$, $p$, $q$ be as in \eqref{range.exp}. Then the solution $\widetilde{H}(x,t)$ evolves with the velocity $\widetilde{c}=kn+2\sqrt{p-q}$. More precisely, we have (with locally uniform convergence)
\begin{equation}\label{convantiH1}
\lim\limits_{t\to\infty}\widetilde{H}(x-ct,t)=\begin{cases}
                                    0, & \mbox{if } c<\widetilde{c} \\
                                    1, & \mbox{if } c>\widetilde{c},
                                  \end{cases}
\end{equation}
and the following convergence in form to the traveling wave with profile $f_{\widetilde{c}}$
\begin{equation}\label{convantiH2}
\lim\limits_{t\to\infty}\sup_{x\in\real}|\widetilde{H}(x,t)-f_{\widetilde{c}}(x+\widetilde{c}t)|=0.
\end{equation}
\end{theorem}
The outcome of Theorem \ref{th.antiHeav} can be explained at an intuitive level by the fact that in this case the convection teams up with the other terms of the equation in order to push the front in the backward direction (that is, towards $-\infty$) with the minimal velocity allowing for the existence of a monotone traveling wave connecting the level one at $-\infty$ to the level zero at $+\infty$; this velocity is exactly $\widetilde{c}$, since for $c<\widetilde{c}$ the critical point $P_2$ becomes an unstable focus and thus any traveling wave with velocity $c<\widetilde{c}$ will present infinitely many (damped) oscillations around the level one. We also observe that, in fact, Theorem \ref{th.antiHeav} is equivalent to Theorem \ref{th.Heav} for convection coefficients $k<0$, since, as explained already in the Introduction, considering negative convection coefficients $k<0$ reduces to Eq. \eqref{eq1} after performing the simple transformation $x=-y\in\real$.

Let us now come back to the critical velocity $\overline{c}$ defined in \eqref{crit.vel}. As explained above, the outcome of Theorem \ref{th.Heav} is due to the convection term $k(u^n)_x$, which competes against the natural movement of the waves generated by the non-convective equation. The following result makes this effect even more striking, showing that a \emph{continuous transition from vanishing to spreading} is achieved through the coefficient $k$ of the convective term, even for fixed exponents $n$, $p$, $q$.
\begin{theorem}\label{th.transition}
Let $n$, $p$, $q$ as in \eqref{range.exp}. Then, there is $k^*(n,p,q)>0$ depending only on $n$, $p$ and $q$ and satisfying
\begin{equation}\label{crit.coef}
\frac{2\sqrt{p-q}}{n}\leq k^*(n,p,q)\leq\max\left\{1,\frac{p}{n}\right\},
\end{equation}
such that $\overline{c}<0$ if $-\infty<k<k^*(n,p,q)$, $\overline{c}=0$ if $k=k^*(n,p,q)$ and $\overline{c}>0$ if $k>k^*(n,p,q)$. If, furthermore,
\begin{equation}\label{cond.conv}
n\leq\frac{q+1}{2},
\end{equation}
then the equality in the lower bound of \eqref{crit.coef} is achieved; that is, $\overline{c}>0$ for any $k$ such that $kn-2\sqrt{p-q}>0$.
\end{theorem}

We plot in Figure \ref{fig:main} the outcome of a numerical experiment showing the variation with respect to time of the Heaviside solution, for all the three relative positions of $k$ with respect to the critical coefficient $k^*(n,p,q)$. The graph in the dashed line corresponds to the explicit solution available for the case $p=n$ and $q=1$ (see Section \ref{subsec.exp}), in order to compare the evolution of the true solution with its asymptotic profile. We mention that the numerical experiment has been performed on a compactified real line by employing the transformation $z=\frac{2}{\pi}\arctan\,x$.

\begin{figure}[ht!]
  \begin{center}
  \subfigure[Case $k<k^*(n,p,q)$: $k=0.5$]{\includegraphics[width=7cm,height=5cm]{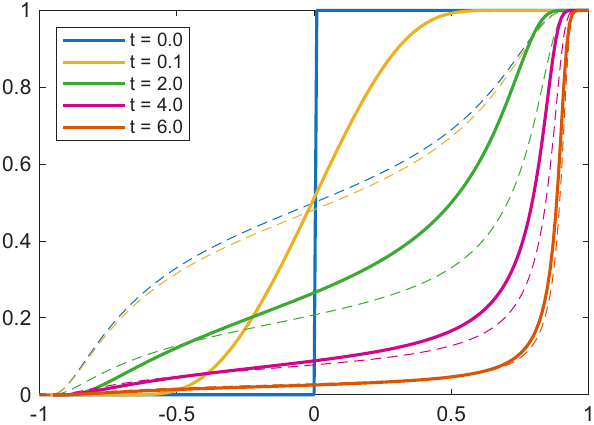}}
  \subfigure[Case $k>k^*(n,p,q)$: $k=2$]{\includegraphics[width=7cm,height=5cm]{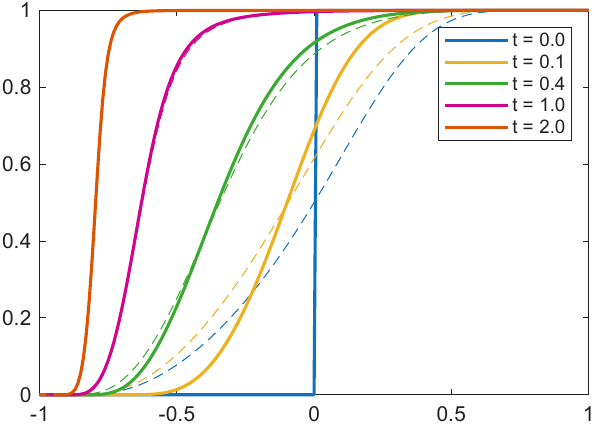}}
  \subfigure[Case $k=k^*(n,p,q)$: $k=1$]{\includegraphics[width=9cm,height=5cm]{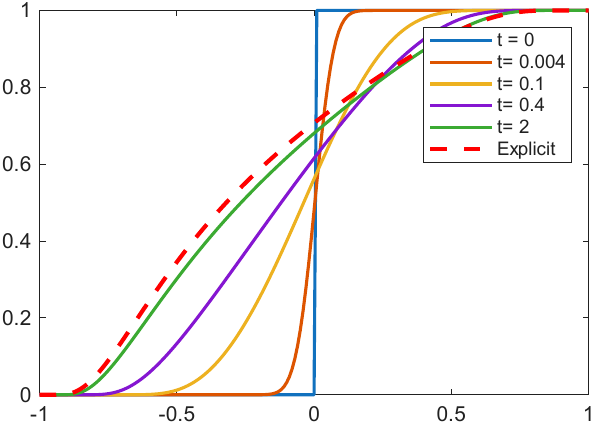}}
  \end{center}
  \caption{Behavior of the solution with initial condition $H_0(x)$ for all the relative positions of $k$ with respect to $k^*(n,p,q)$: experiment for $n=3$, $p=3$, $q=1$, with $k^*(n,p,q)=1$}\label{fig:main}
\end{figure}

It thus follows that both the coefficient $k$ and the exponent $n$ modulate the strength of the convection. Indeed, the natural movement determined by the reaction, absorption and diffusion terms is to converge to zero, while the convection term produces an evolution of the wavefronts in the opposite direction (towards the spreading regime), creating thus a competition with the effect generated by the other terms. It is on the one hand obvious that, as larger the coefficient $k$, as stronger the effect of the convection term becomes. Note that we have also considered coefficients $k<0$ in the statement of Theorem \ref{th.transition}, since the result is the same: when $k<0$ the convection changes direction and pushes even more the solutions towards the vanishing regime.

On the other hand, since $(u^n)_x=nu^{n-1}u_x$, if the convection exponent $n$ is sufficiently small (in particular, satisfying \eqref{cond.conv}), then the factor $u^{n-1}$ remains sufficiently large even for small values of $u$, enhancing the effect of the convection and thus favoring the faster transport of mass. This intuitively explains the fact that the convection tops the above mentioned competition and spreads the waves with a positive velocity $\overline{c}$. In the opposite range $n>(q+1)/2$, the critical coefficient $k^*(n,p,q)$ becomes a threshold between the vanishing and spreading regimes and a number of estimates (some of them being sharp) of $k^*(n,p,q)$ are given in Section \ref{sec.trans} the paper when $n$, $p$, $q$ are related in suitable ways. Let us mention here that, in the case $p=n$, $q=1$ and any $k>0$ such that $k^2>1/(n-1)$, we are able to obtain the explicit value of the critical velocity as $\overline{c}=(k^2-1)/k$, whose sign changes at $k^*(n,n,1)=1$, see the beginning of Section \ref{subsec.exp}.

We end this discussion by observing that the threshold convection coefficient $k^*(n,p,q)$ derived in Theorem \ref{th.transition} plays a role that is analogous to the transition parameter separating pulled and pushed propagation regimes in Burgers-Fisher-KPP type equations. In both settings, the coefficient in front of the convective term governs a qualitative change in the selection mechanism of the asymptotic front speed. As commented before, in our generalized Burgers-Fisher-KPP equation $k^*(n,p,q)$ marks the threshold between negative and positive critical velocities, that is, between vanishing and spreading regimes, reflecting whether convection is strong enough to overcome the diffusive-reactive dynamics. This is conceptually parallel to the pulled to pushed transition described in works such as \cite{AHR25, AH26, EvS00, GGHR12} (see also references therein), where the front speed changes from being linearly determined (pulled) to nonlinearly selected (pushed) once the convective coefficient exceeds a critical threshold. In both frameworks, the transition corresponds to a change in the phase space structure when selecting the admissible heteroclinic connections, and thus the relevant traveling waves.

\section{Construction of sub- and supersolutions}

In this preparatory section, we construct suitable sub- and supersolutions to Eq. \eqref{eq1} by truncating suitable traveling waves with the constant solutions zero or one. These sub- and supersolutions are the main tool in the proofs of the main results and will be constructed employing the system \eqref{PSsyst}. Note that, due to Duhamel's formula, solutions to Eq. \eqref{eq1} are classical, but for sub- and supersolutions we employ the weak formulation, since we want to allow them to have a possible jump of the first derivative.
\begin{definition}\label{def.subsuper}
We say that a function $u\in L^{\infty}(\real)\cap C(\real)$ such that $u_x\in L^1_{\rm loc}(\real\times(0,\infty))$ is a subsolution to Eq. \eqref{eq1} if the following inequality
\begin{equation*}
\int_{\real}(u(t)\varphi(t)-u(0)\varphi(0))\,dx\leq\int_0^t\int_{\real}[u\varphi_t-u_x\varphi_x+ku^n\varphi_x+(u^p-u^q)\varphi]\,dx\,ds
\end{equation*}
holds true for any function $\varphi\in C_0^{\infty}(\real\times[0,\infty))$ such that $\varphi\geq0$. We say that $u$ is a supersolution if the inequality sign is reversed.
\end{definition}
Employing the weak formulation allows us to consider sub- and supersolutions that are only continuous at some point, but not differentiable, as the ones constructed in Lemmas \ref{lem.sub} and \ref{lem.super} given in the forthcoming lines.

Consider first $c\in(-\infty,kn-2\sqrt{p-q})$, a range in which the critical point $P_2$ is a stable node. In this range, it is obvious that $0>\lambda_+(c)>\lambda_-(c)$, where $\lambda_+(c)$ and $\lambda_-(c)$ are the eigenvalues of the linearization of the system in a neighborhood of $P_2$, defined in \eqref{eigen.P2}. By standard theory of dynamical systems (see \cite[Theorem 4.3, Chapter 2]{Z4}), all but one of the trajectories connecting to $P_2$ from the positive cone $(0,\infty)^2$ reach this critical point tangent to the eigenvector $e_{+}(c)=(1,\lambda_+(c))$, and a single trajectory reaches it tangent to $e_{-}(c)=(1,\lambda_-(c))$. We recall here that, for any $c\in\real$, $l_c$ denotes the unique trajectory of the system \eqref{PSsyst} with velocity parameter $c$ going out of the critical point $P_1$.
\begin{lemma}\label{lem.sub}
Let $n$, $p$, $q$, $k$ be as in \eqref{range.exp} and $c\in(-\infty,kn-2\sqrt{p-q})$ be such that the trajectory $l_c$ of the system \eqref{PSsyst} connects directly $P_1$ to $P_2$ in the sense of Definition \ref{def.con} and tangent to the eigenvector $e_{+}(c)$. Then, for any $R>0$, there exists a subsolution to Eq. \eqref{eq1} in the form of a traveling wave with velocity $c$,
\begin{equation*}
\underline{U}_{c,R}(x,t)=\underline{f}_{c,R}(x+ct), \quad (x,t)\in\real\times(0,\infty)
\end{equation*}
such that $0\leq\underline{U}_{c,R}<1$ and $\underline{U}(x,0)=0$ for $x\in(-\infty,R]$.
\end{lemma}
\begin{proof}
Let $c$ be as in the statement and let $L$ be the unique trajectory of the system \eqref{PSsyst} connecting to $P_2$ tangent to the eigenvector $e_{-}(c)$. Let $(X_1(\xi),Y_1(\xi))$ be the parametrization of $L$ in the system \eqref{PSsyst}. It follows that
$$
\lim\limits_{\xi\to\infty}X_1(\xi)=1, \quad \lim\limits_{\xi\to\infty}Y_1(\xi)=0.
$$
We prove that there is $\xi_0\in\real$ such that $X_1(\xi_0)=0$, $Y_1(\xi_0)>0$, that is, the trajectory $L$ crosses the $Y$-axis. From the first equation of the system \eqref{PSsyst}, we deduce that $X'(\xi)>0$ whenever $(X(\xi),Y(\xi))\in(0,\infty)^2$. Assume for contradiction that the trajectory $L$ does not intersect the $Y$-axis. Consider next the region limited by the vertical lines $X=0$, $X=1$ and the curve $l_c$, that is,
$$
\mathcal{S}:=\{(X,Y)\in\real^2: 0<X<1, Y>Y_c(X)\},
$$
where $Y_c(X)$ is the parametrization of the curve $l_c$ obtained by inverting the increasing function $\xi\mapsto X(\xi)$. Since the trajectory $l_c$ connecting directly $P_1$ to $P_2$ is a separatrix of the system \eqref{PSsyst} and $\lambda_-(c)<\lambda_+(c)<0$, it follows that there is $\xi_0\in\real$ such that $(X_1(\xi),Y_1(\xi))\in\mathcal{S}$ for any $\xi>\xi_0$. The monotonicity of $X_1$ and the assumption that the trajectory $L$ does not cross the $Y$-axis imply that the trajectory $L$ remains forever in the region $\mathcal{S}$. Thus, the inverse function theorem allows to parametrize the curve $L$ as the graph of a function $Y_1=Y_1(X_1)$, whose derivative is
\begin{equation}\label{interm1}
\begin{split}
Y_1'(X_1)&=\frac{cY_1(X_1)-knX_1^{n-1}Y_1(X_1)-X_1^p+X_1^q}{Y_1(X_1)}=c-knX_1^{n-1}+\frac{X_1^q-X_1^p}{Y_1(X_1)}\\
&\geq c-knX_1^{n-1},
\end{split}
\end{equation}
the bound from below in \eqref{interm1} following from the fact that we are under the assumption that $X_1\in[0,1]$ along the trajectory $L$. It follows easily from \eqref{interm1} that there is no $\widetilde{X}\in[0,1]$ such that the trajectory $L$ has a vertical asymptote in the sense that
$$
\lim\limits_{X_1\to\widetilde{X}^+}Y_1(X_1)=+\infty.
$$
Indeed, if this was the case, then \eqref{interm1} gives that $\lim\limits_{X_1\to\widetilde{X}^+}Y_1'(X_1)=c-kn\widetilde{X}^{n-1}$, contradicting the vertical asymptote at $\widetilde{X}$. The previous argument ensures that the whole trajectory $L$ remains in a compact set, while the bound from below of the derivative $Y_1'(X_1)$ in \eqref{interm1} prevents the trajectory $L$ for oscillating infinitely many times inside the compact set. The monotonicity of the coordinate $X$, the uniqueness of the unstable trajectory from $P_1$, the lack of other critical points in the strip $0\leq X<1$ and an application of the Poincar\'e-Bendixon's Theorem (see \cite[Section 3.7]{Pe}) lead to the non-existence of an $\alpha$-limit for the trajectory $L$ and thus a contradiction, which stems from our initial assumption that $L$ does not cross the $Y$-axis. It follows that there is $\xi_0\in\real$ such that $X_1(\xi_0)=0$ and $Y_1(\xi_0)>0$. By undoing the change of variable \eqref{PSchange}, we infer that there is a traveling wave profile $\widetilde{f}_c$ such that
$$
\widetilde{f}_c(\xi_0)=0, \quad \widetilde{f}'_{c}(\xi_0)>0, \quad \widetilde{f}_c(\xi)\in(0,1) \ {\rm for} \ \xi\in(\xi_0,\infty).
$$
We then extend $\widetilde{f}_c$ by zero for $\xi\in(-\infty,\xi_0)$ and translate in space to define
$$
\underline{U}_{c,R}(x,t)=\widetilde{f}_{c}(x+\xi_0-R+ct),
$$
which satisfies the claimed property.
\end{proof}
The following consequence is important in some of the forthcoming proofs.
\begin{corollary}\label{cor.sub}
There exist subsolutions as in Lemma \ref{lem.sub} corresponding to profiles connecting to $P_2$ but tangent to the eigenvector $e_{+}(c)$.
\end{corollary}  
\begin{proof}
Recalling that $L$ is the unique trajectory of the system \eqref{PSsyst} connecting to $P_2$ tangent to the eigenvector $e_{-}(c)$, we consider the open region $\mathcal{T}$ limited by the trajectories $l_c$, $L$ and the axis $X=0$. We then easily deduce by similar arguments as in the proof of Lemma \ref{lem.sub} that any trajectory of the system \eqref{PSsyst} passing through a point in $\mathcal{T}$ ends in the critical point $P_2$ and, in the backward direction, crosses the axis $X=0$. The uniqueness of the trajectory tangent to $e_{-}(c)$ implies that all these trajectories enter $P_2$ tangent to the leading eigenvector $e_{+}(c)$.
\end{proof}
The next lemma is aimed at constructing suitable supersolutions to Eq. \eqref{eq1}.
\begin{lemma}\label{lem.super}
Let $c\in\real$ be such that the trajectory $l_c$ intersects for the first time the $X$-axis at a point $(X_0,0)$ with $X_0\neq1$. Then, for any $R>0$, there exists a supersolution in the form of a traveling wave
\begin{equation*}
\overline{U}_{c,R}(x,t)=\overline{f}_c(x+ct), \quad (x,t)\in\real\times(0,\infty)
\end{equation*}
such that $0<\overline{U}_{c,R}\leq 1$ and $U_{c,R}(x,0)=1$ for any $x\in[R,\infty)$.
\end{lemma}
\begin{proof}
We infer from the direction of the flow of the system \eqref{PSsyst} across the axis $Y=0$ that $X_0>1$. Let $\xi_0\in\real$ be such that $(X(\xi_0),Y(\xi_0))=(X_0,0)$. Since $\lim\limits_{\xi\to-\infty}X(\xi)=0$ on the trajectory $l_c$, it follows from the monotonicity of the mapping $\xi\mapsto X(\xi)$ that there is a unique $\xi_1\in(-\infty,\xi_0)$ such that $X(\xi_1)=1$ and $Y(\xi_1)=Y_1>0$. Let $f_c$ be the profile of a traveling wave corresponding to the trajectory $l_c$ by undoing the change of variable \eqref{PSchange}. It follows that $f_c(\xi_1)=1$ and $f_c'(\xi_1)>0$. We thus extend $f_c$ by one in the interval $(\xi_1,\infty)$, that is, introduce
$$
\overline{f}_c(\xi)=\begin{cases}
                    f_c(\xi), & \mbox{if } \xi\in(-\infty,\xi_1) \\
                    1, & \mbox{if } \xi\in[\xi_1,\infty)
                  \end{cases}
$$
and employ a translation in space to define the super-solution
$$
\overline{U}_{c,R}(x,t)=\overline{f}_c(x+\xi_1-R+ct),
$$
which fulfills the claimed properties, completing the proof.
\end{proof}
We end this section with one more preparatory result describing the trajectory $l_{\overline{c}}$ corresponding to the critical velocity $\overline{c}$ defined in \eqref{crit.vel}.
\begin{lemma}\label{lem.crit}
Let $\overline{c}$ be the critical velocity defined in \eqref{crit.vel}. Then the trajectory $l_{\overline{c}}$ connects the critical points $P_1$ and $P_2$, arriving at $P_2$ tangent to the eigenvector $e_{-}(\overline{c})$ corresponding to the non-leading eigenvalue $\lambda_{-}(\overline{c})$ in \eqref{eigen.P2}.
\end{lemma}
\begin{proof}
The monotone increasing character of the mapping $\xi\mapsto X(\xi)$ along any trajectory $l_c$ while $Y(\xi)>0$ allows us to parametrize $l_c$ as the graph of a function $Y=Y_c(X)$ for any $X\in[0,X_0(c)]$, where $(X_0(c),0)$ is the first point where the trajectory $l_c$ intersects the $X$-axis (which can be the critical point $P_2$ if $X_0(c)=1$ or a regular point if $X_0(c)\in(1,\infty)$). We first recall the following monotonicity result established in \cite[Proposition 4.1]{IS26}: $X_0(c)<\infty$ for any $c\in\real$ and, if $c_1$, $c_2\in\real$ are such that $c_1<c_2$, then $Y_{c_1}(X)<Y_{c_2}(X)$ for any $X\in(0,X_0(c_1))$. Consider next the following three sets:
\begin{equation*}
\begin{split}
&\mathcal{A}=\{c\in\real: l_c \ {\rm connects \ directly} \ P_1 \ {\rm to} \ P_2 \ {\rm tangent \ to} \ e_{+}(c)=(1,\lambda_+(c))\},\\
&\mathcal{C}=\{c\in\real: l_c \ {\rm crosses \ the} \ X-{\rm axis} \ {\rm at} \ (X_0(c),0), \ X_0(c)>1\},\\
&\mathcal{B}=\real\setminus(\mathcal{A}\cup\mathcal{B}).
\end{split}
\end{equation*}
It is obvious that for $c\in\mathcal{B}$ (if any), then the trajectory $l_c$ connects directly $P_1$ to $P_2$ but with tangency to the eigenvector $e_{-}(c)=(1,\lambda_{-}(c))$. It follows from Proposition \ref{prop.P2} and \cite[Proposition 4.1]{IS26} that $\mathcal{C}$ is non-empty and contains at least the interval $(kn-2\sqrt{p-q},\infty)$, while $\mathcal{C}$ is open by easy arguments related to the continuity of the family of systems \eqref{PSsyst} with respect to the velocity parameter $c$ (see for example \cite[Lemma 7.1]{IS26} for a detailed proof). We next prove that $\mathcal{A}$ is non-empty and open. We readily obtain from \eqref{eigen.P2} that $c\mapsto\lambda_{-}(c)$ is an increasing function for $c\in(-\infty,kn-2\sqrt{p-q})$ and that
$$
-\infty=\lim\limits_{c\to-\infty}\lambda_{-}(c)<\lambda_{-}(c)\leq -2\sqrt{p-q}=\lambda_{-}(kn-2\sqrt{p-q})<0.
$$
We know from \cite[Proposition 6.1]{IS26} that $(-\infty,-2\sqrt{p-q})\subseteq\mathcal{A}\cup\mathcal{B}$. Assume for contradiction that there is some $c_0\in(-\infty,-2\sqrt{p-q})$ such that the trajectory $l_{c_0}$ enters $P_2$ tangent to the eigenvector $e_{-}(c_0)=(1,\lambda_{-}(c_0))$. Then, the opposite monotone behavior of the slope of the eigenvector $e_{-}(c)$ (decreasing in absolute value with respect to $c$) and of the trajectory $l_c$ (increasing with respect to $c$) immediately gives that, for $c>c_0$, the trajectory $l_c$ does no longer enter $P_2$ directly but crosses the $X$-axis at a point $(X_1(c),0)$ with $X_1(c)>1$, which is a contradiction. We have just proved that $(-\infty,-2\sqrt{p-q})\subseteq\mathcal{A}$. Moreover, the same argument shows that $\mathcal{A}$ is an open set. It follows from the same opposite monotonicity that $\mathcal{B}$ contains exactly one element and this element is exactly $\overline{c}$ defined in \eqref{crit.vel}, as claimed.
\end{proof}
We end this section with a theoretical preparatory result related to dynamical systems that we employ several times throughout the paper.
\begin{lemma}\label{lem.flow}
Let $\Gamma$ be a curve in the $(X,Y)$ plane associated with the system \eqref{PSsyst} such that it is the graph of a function $Y=g(X)$ with $g(0)=0$, $g(1)=0$ and $g(X)>0$ for $X\in(0,1)$.

(a) Assume that the direction of the flow of the system \eqref{PSsyst} across $\Gamma$ points towards the closed region $\mathcal{D}$ limited by $\Gamma$ and the $X$-axis. Then, the unique unstable trajectory going out of $P_1$ enters the region $\mathcal{D}$ and connects directly to $P_2$.

(b) On the contrary, if the direction of the flow of the system \eqref{PSsyst} across $\Gamma$ points outside the closed region $\mathcal{D}$, then the unique unstable trajectory going out of $P_1$ goes out in the complementary region to $\mathcal{D}$.
\end{lemma}
\begin{proof}
(a) Note first that $\Gamma$ connects $P_1$ to $P_2$ and is contained in the strip $[0,1]\times[0,\infty)$. Assume for contradiction that the unique trajectory $l_c$ contained in the unstable manifold of $P_1$ (according to Propositions \ref{prop.P1}, \ref{prop.P1c0} and \ref{prop.P1q1}) goes out from $P_1$ outside the region $\mathcal{D}$. According to \cite[Propositions 3.2 and 4.1]{IS26}, the trajectory $l_c$ either connects directly to the critical point $P_2$, or crosses the $X$-axis at a point $(X_1(c),0)$ with $X_1(c)>1$. In both cases, a different closed region $\mathcal{D}_1$ limited by the trajectory $l_c$, the curve $\Gamma$ and, if the latter case is in force, the interval $I(c):=\{X\in[1,X_1(c)], Y=0\}$ of the $X$-axis, is generated.

Pick a point $(X_0,Y_0)$ situated in the interior of the region $\mathcal{D}_1$ with $X_0\in[0,1]$. Since this point is not an equilibrium one, there is a unique trajectory of the system \eqref{PSsyst} passing through it. Note that the boundary of the region $\mathcal{D}_1$ is composed by: the curve $\Gamma$ whose direction of the flow points towards the exterior of $\mathcal{D}_1$, the trajectory $l_c$ (which is a separatrix of the phase plane) and, possibly, the interval $I(c)$ defined above, on which the direction of the flow of the system \eqref{PSsyst} also points towards the negative half-plane $Y<0$ (as it can be very easily checked by the sign of $X^q-X^p$ when $X>1$). Since the region $\mathcal{D}_1$ is a compact set, an application of the Poincar\'e-Bendixon's Theorem \cite[Section 3.7]{Pe} in the opposite direction of the flow on the trajectory passing through the fixed point $(X_0,Y_0)\in\mathcal{D}_1$ gives that the only possibility for the $\alpha$-limit of this trajectory is the critical point $P_1$. However, this is a contradiction with the uniqueness of the unstable trajectory from $P_1$.

We have thus shown that the trajectory $l_c$ stemming from $P_1$ enters the region $\mathcal{D}$. Moreover, the direction of the flow of the system \eqref{PSsyst} across the interval $\{X\in[0,1], Y=0\}$ of the $X$-axis points into the interior of $\mathcal{D}$, which, together with the condition on $\Gamma$, ensures that $\mathcal{D}$ is a positively invariant region. The compactness of $\mathcal{D}$, the monotonicity of $X$ along the trajectory $l_c$ and the Poincar\'e-Bendixon's Theorem then imply that $l_c$ has to connect to the critical point $P_2$, completing the proof of part (a). The proof of part (b) follows by completely similar arguments.
\end{proof}

\section{Proof of Theorems \ref{th.Heav} and \ref{th.antiHeav}}\label{sec.Heav}

We are now in a position to employ the sub- and supersolutions constructed in Lemmas \ref{lem.sub} and \ref{lem.super} in order to prove Theorem \ref{th.Heav}. The proof combines comparison techniques in order to establish the critical velocity $\overline{c}$, together with the technique used in \cite{KPP37} in order to prove the convergence of the frontwave to the traveling wave with profile $f_{\overline{c}}$. We first need a general preparatory result.
\begin{lemma}\label{lem.monot}
Let $u_0$ be such that $0\leq u_0(x)\leq 1$ for any $x\in\real$. If $u_0$ is non-decreasing (respectively non-increasing), then the solution $u$ to the Cauchy problem for Eq. \eqref{eq1} with initial condition $u_0$ is non-decreasing (respectively non-increasing) with respect to $x\in\real$ as well, for any $t>0$.
\end{lemma}
\begin{proof}
The comparison principle gives $0<u(x,t)<1$ for any $(x,t)\in\real\times(0,\infty)$. Assume first that $u_0$ is two times differentiable and let $v(x,t):=u_x(x,t)$. We compute the equation solved by $v$ by differentiating Eq. \eqref{eq1} with respect to $x$ to find
\begin{equation}\label{interm4}
v_t=v_{xx}+knu^{n-1}v_x+kn(n-1)u^{n-2}v^2+(pu^{p-1}-qu^{q-1})v.
\end{equation}
Since $n\geq2$, $p>1$ and $q\geq1$, the coefficients of \eqref{interm4} are bounded. Moreover, $v\equiv0$ is a solution to \eqref{interm4}. Thus, the comparison principle for semilinear parabolic equations ensures that $v(t)\geq0$ for any $t>0$ whenever $v_0\geq0$ (respectively, $v(t)\leq0$ for any $t>0$ whenever $v_0\leq0$), completing the proof for differentiable initial conditions. The general result (removing the assumption of differentiability) follows by a standard approximation argument of $u_0$ by functions $u_{0,n}\in C^2(\real)$ preserving the same monotonicity as $u_0$, the convergence of the derivatives of the approximating solutions at any $t>0$ to the actual solution being ensured by the uniform estimates in \cite[Chapter V, Section 8]{LSU}.
\end{proof}
\begin{proof}[Proof of Theorem \ref{th.Heav}]
We divide the proof in a few steps, for the readers' convenience.

\medskip

\noindent \textbf{Step 1. Critical velocity.} In this first step, we prove \eqref{convH1}. Pick $c<\overline{c}$. We infer from Lemma \ref{lem.crit} that $c\in\mathcal{A}$ and thus the trajectory $l_c$ connects $P_1$ to $P_2$ tangent to the eigenvector $e_{+}(c)=(1,\lambda_+(c))$. It thus follows from Lemma \ref{lem.sub} that there exists a traveling wave subsolution $\underline{U}_{c,0}$ such that $0\leq\underline{U}_{c,0}\leq1$ and $\underline{U}_{c,0}(x,0)=0$ for any $x\in(-\infty,0]$. The comparison principle then gives that
\begin{equation*}
H(x,t)\geq\underline{U}_{c,0}=\underline{f}_c(x+ct), \quad (x,t)\in\real\times(0,\infty),
\end{equation*}
which also gives
\begin{equation}\label{interm2}
H(x-ct,t)\geq\underline{f}_{c}(x), \quad (x,t)\in\real\times(0,\infty).
\end{equation}
Letting next $c'\in(c,\overline{c})$, we apply \eqref{interm2} for the velocity $c'$ to deduce that
$$
H(x-ct,t)=H(x+(c'-c)t-c't,t)\geq\underline{f}_{c'}(x+(c'-c)t).
$$
Since $c'<\overline{c}$, the trajectory $l_{c'}$ also directly connects $P_1$ to $P_2$ and thus the corresponding subsolution satisfies
$$
\lim\limits_{\xi\to\infty}f_{c'}(\xi)=1.
$$
Fixing $x\in\real$, letting $t\to\infty$ and taking into account that $c'-c>0$, we get that
$$
\lim\limits_{t\to\infty}H(x-ct,t)\geq\lim\limits_{t\to\infty}\underline{f}_{c'}(x+(c'-c)t)=1,
$$
while the opposite inequality follows from the fact that $0\leq H(x,t)\leq1$ for any $(x,t)\in\real\times(0,\infty)$. The previous proof applies for any $c<\overline{c}$, and we have thus shown the first statement of \eqref{convH1}.

Let now $c>\overline{c}$. Then, Lemma \ref{lem.crit} gives that $c\in\mathcal{C}$ and thus there is a supersolution $\overline{U}_{c,0}$ such that $0\leq\overline{U}_{c,0}\leq 1$ and $\overline{U}_{c,0}(x,0)=1$ for any $x\in[0,\infty)$. The comparison principle then ensures that
$$
H(x,t)\leq\overline{U}_{c,0}(x,t)=\overline{f}_{c}(x+ct), \quad (x,t)\in\real\times(0,\infty),
$$
or, equivalently,
\begin{equation}\label{interm3}
H(x-ct,t)\leq\overline{f}_{c}(x), \quad (x,t)\in\real\times(0,\infty).
\end{equation}
Pick now $\hat{c}\in(\overline{c},c)$. By applying \eqref{interm3} for the velocity $\hat{c}$, we obtain
$$
H(x-ct,t)=H(x+(\hat{c}-c)t-\hat{c}t,t)\leq\overline{f}_{\hat{c}}(x+(\hat{c}-c)t).
$$
Taking into account that
$$
\lim\limits_{\xi\to-\infty}\overline{f}_{\hat{c}}(\xi)=0
$$
and the fact that $x+(\hat{c}-c)t\to-\infty$ for any $x$ fixed and $t\to\infty$ (since $\hat{c}<c$), we deduce that
$$
0\leq\lim\limits_{t\to\infty}H(x-ct,t)\leq\lim\limits_{t\to\infty}\overline{f}_{\hat{c}}(x+(\hat{c}-c)t)=0
$$
and thus the second convergence result in \eqref{convH1} is proved. Finally, in both cases the convergence in \eqref{convH1} is uniform for $x$ in any compact set of $\real$, this fact following from the obvious locally uniform convergence as $t\to\infty$ of the profiles $\underline{f}_{c'}(x+(c'-c)t)$, respectively $\overline{f}_{\hat{c}}(x+(\hat{c}-c)t)$.

\medskip

\noindent \textbf{Step 2. Linearization.} We are left with the proof of the precise convergence to the traveling wave denoted by $f_{\overline{c}}$, for which we repeat practically verbatim the proof in the classical paper \cite{KPP37}. To this end, the main tool is the fact that the difference of two solutions is a solution to an equation satisfying a maximum principle. Indeed, if $H(x,t)$ is the solution stemming from the initial condition $H_0(x)$, it follows that $0\leq H(t)\leq1$ for any $t>0$. Moreover, since exponents $n$, $p$, $q$ satisfy \eqref{range.exp}, the function (employing the notation in \cite{LSU})
$$
a(x,t,u,u_x)=-knu^{n-1}u_x-u^p+u^q
$$
satisfies the assumptions in \cite[Theorem 8.1, Chapter V]{LSU} and thus $H$ is a classical solution to Eq. \eqref{eq1} in $\real\times(0,\infty)$ which is uniformly bounded, together with its derivatives up to second order; in particular
\begin{equation}\label{interm5}
H_x\in L^{\infty}(\real\times(\tau,\infty)), \quad \tau>0.
\end{equation}
Defining, for some constants $C>0$ and $t_0>0$
\begin{equation}\label{interm6}
w(x,t):=H(x,t)-H(x+C,t+t_0)=H(x,t)-H_{C,t_0}(x,t),
\end{equation}
we obtain by direct calculation that
\begin{equation}\label{interm7}
\begin{split}
w_t(x,t)&=H_{xx}(x,t)-H_{xx}(x+C,t+t_0)\\&+kn[(H^{n-1}H_x)(x,t)-(H^{n-1}H_x)(x+C,t+t_0)]\\&+(H^p-H^q)(x,t)-(H^{p}-H^q)(x+C,t+t_0)\\
     & =w_{xx}+knH^{n-1}(x+C,t+t_0)w_x+A(x,t)w,
\end{split}
\end{equation}
where
$$
A(x,t):=\left[\frac{H^p-H_{C,t_0}^p}{H-H_{C,t_0}}-\frac{H^q-H_{C,t_0}^q}{H-H_{C,t_0}}+kn\frac{H^{n-1}-H_{C,t_0}^{n-1}}{H-H_{C,t_0}}H_x\right](x,t).
$$
Taking into account the bound \eqref{interm5} and that $n\geq2$, $p>1$, $q\geq1$, we deduce that $A\in L^{\infty}(\real\times(\tau,\infty))$ for any $\tau>0$, and the standard theory for linear parabolic equations ensures that the equation \eqref{interm7} satisfied by $w$ fulfills the strong maximum principle.

\medskip

\noindent \textbf{Step 3. Existence of a limit.} From the previous two steps, we have all the ingredients employed in the classical proof of the convergence in \cite{KPP37}, and the rest of the proof can be performed exactly as there. We give here a brief sketch of the proof, for the sake of completeness, avoiding the technical steps that can be done exactly as in the quoted classical reference. The main step in the proof (see \cite[Th\'eor\`eme 14]{KPP37}) is, for any $t_0>0$ fixed, to introduce the function
\begin{equation*}
H_{t_0}(x,t)=H(x+\varphi(t_0),t+t_0),
\end{equation*}
where the function $\varphi(t_0)$ is defined such that $H_{t_0}(0,0)=H(\varphi(t_0),t_0)$ has the same value for any $t_0>0$. Fixing $t\in[0,T]$ with $T>0$ arbitrary, it is proved in \cite[Th\'eor\`eme 14]{KPP37} that there exists
$$
\lim\limits_{t_0\to\infty}H_{t_0}(x,t)=\overline{H}(x,t), \quad (x,t)\in\real\times(0,T),
$$
and $\overline{H}$ is a solution to the same equation \eqref{eq1}. In order to prove the existence of the previous limit, we consider the function
$$
w_{t_0}(x,t):=H_{t_0}(x,t)-H_{t_0+T}(x,t),
$$
and we deduce from Step 2 that $w_{t_0}$ is a solution to an equation of the form \eqref{interm7}, satisfying the strong maximum principle. Moreover, it can be proved identically as in \cite[Th\'eor\`eme 13]{KPP37} (whose proof only depends on the strong maximum principle and the inverse function theorem) that, for any $\epsilon>0$, there is $T(\epsilon)>0$ sufficiently large such that $|w_{t_0}(x,0)|<\epsilon$ for $x\in\real$ and $t_0>T(\epsilon)$. Recalling that the coefficient $A(x,t)$ of the linear zero order term in \eqref{interm7} is uniformly bounded, we deduce by comparison with supersolutions depending only on the time variable, of the form
$$
\overline{w}(t)=\epsilon e^{\|A(x,t)\|_{\infty}t},
$$
that
$$
|w_{t_0}(x,t)|\leq\epsilon e^{\|A(x,t)\|_{\infty}t}, \quad (x,t)\in\real\times(0,T), \quad {\rm provided} \ t_0>T(\epsilon).
$$
Note that the convection term, which is the only difference with respect to the analysis in \cite{KPP37}, does not bring any difference at this point, since the term $knH^{n-1}(x+C,t+t_0)w_{x}$ in \eqref{interm7} vanishes when applied to the supersolution depending only on the time variable. We have thus shown that
$$
|H_{t_0}(x,t)-H_{t_0+T}(x,t)|\leq\epsilon e^{\|A(x,t)\|_{\infty}t}, \quad (x,t)\in\real\times(0,T), \quad {\rm provided} \ t_0>T(\epsilon),
$$
which implies the uniform convergence of $H_{t_0}(x,t)$ to a limit $\overline{H}(x,t)$ as $t_0\to\infty$ in $\real\times[0,T]$. The fact that $\overline{H}$ is a solution to Eq. \eqref{eq1} follows by passing to the limit as $t_0\to\infty$ in the weak formulation of Eq. \eqref{eq1}, employing the uniform convergence in order to pass to the limit in the integrals. One can show as well that the first order derivatives of $H_{t_0}$ converge to first order derivatives of $\overline{H}$ by employing the uniform boundedness of the derivatives up to second order according to \cite[Theorem 8.1, Chapter V]{LSU} and the Arzel\'a-Ascoli Theorem.

\medskip

\noindent \textbf{Step 4. Identification of the limit and end of the proof.} We are only left to identify the limit $\overline{H}$ obtained in the previous step. In order to do this, we proceed as in \cite[Th\'eor\`eme 17]{KPP37} and define the function
$$
H^{*}(x,t):=\overline{H}(x-c(t),t),
$$
where $c(t)$ is implicitly defined by the condition that $H^{*}(0,t)=\overline{H}(-c(t),t)$ is constant for any $t>0$. By differentiating with respect to $t$ and taking into account that $\overline{H}$ is a solution to Eq. \eqref{eq1}, we deduce that
$$
\frac{\partial H^*}{\partial t}(x,t)=\frac{\partial\overline{H}}{\partial t}(x-c(t),t)-c'(t)\frac{\partial\overline{H}}{\partial x}(x-c(t),t),
$$
whence $H^*$ is a solution to the equation
\begin{equation}\label{interm18}
(H^*)_t=(H^*)_{xx}+kn(H^*)^{n-1}(H^*)_x-c'(t)(H^*)_x+(H^*)^p-(H^*)^q.
\end{equation}
Taking into account the construction of $\overline{H}$ as limits of time-translations of $H(\cdot,t)$, we observe that, for any $s>0$,
\begin{equation*}
\begin{split}
H^*(x,t+s)&=\overline{H}(x-c(t+s),t+s)=\lim\limits_{\tau\to\infty}H(x-c(t+s+\tau),t+s+\tau)\\
&=\lim\limits_{\zeta\to\infty}H(x-c(t+\zeta),t+\zeta)=\overline{H}(x-c(t),t)=H^*(x,t),
\end{split}
\end{equation*}
which proves that $H^*$ is a function that does not depend on the second variable; that is, $H^*(x,t)=f^*(x)$. We infer that $(H^*)_t=0$ and \eqref{interm18} then yields that $c'(t)$ is a constant (otherwise, the right-hand side would depend on $t$ explicitly). The convergence in \eqref{convH1} entails that the only possible constant for the velocity $c'(t)$ is $\overline{c}$. We then conclude that $c(t)=\overline{c}t$ and then 
$$
\overline{H}(x-\overline{c}t,t)=f^*(x), \quad (x,t)\in\real\times(0,\infty),
$$ 
whence $\overline{H}(x,t)=f^*(x+\overline{c}t)$ is the only traveling wave solution with velocity $\overline{c}$, completing the proof.
\end{proof}
We move now to the proof of Theorem \ref{th.antiHeav}, which will be rather analogous to the proof of Theorem \ref{th.Heav}. Recall the notation $\tilde{c}=kn+2\sqrt{p-q}$.
\begin{proof}[Proof of Theorem \ref{th.antiHeav}]
Let first $c>\tilde{c}$. It follows from Proposition \ref{prop.P2} that $P_2$ is an unstable node for the system \eqref{PSsyst} with such values of the velocity $c$, and, moreover, it has been proved in \cite[Proposition 3.2]{IS26} that there is $m>0$ (depending on $c$) such that the triangular region limited by the two coordinate axis and the line $Y=m(X-1)$ is negatively invariant with respect to the trajectories of the system \eqref{PSsyst}. We infer that there exists a trajectory going out of $P_2$ and crossing the $Y$-axis at a point $(0,Y_0)$ with $-m<Y_0<0$. The traveling wave profiles $\underline{f}_{c,\xi_0}$ corresponding to this trajectory satisfy that
$$
\lim\limits_{\xi\to-\infty}\underline{f}_{c,\xi_0}(\xi)=1, \quad \underline{f}_{c,\xi_0}(\xi)>0 \ {\rm for} \ \xi\in(-\infty,\xi_0), \quad \underline{f}_{c,\xi_0}(\xi_0)=0,
$$
for some $\xi_0\in\real$. By extending this profile by setting $\underline{f}_{c,\xi_0}\equiv0$ on $(\xi_0,\infty)$, it follows that $\underline{U}(x,t)=\underline{f}_{c,\xi_0}(x+ct)$ is a subsolution to \eqref{eq1}. Picking the interface point $\xi_0\leq0$, we deduce that
$$
\widetilde{H}(x,0)=1-H_0(x)\geq\underline{U}(x,0)=\underline{f}_{c,\xi_0}(x), \quad x\in\real.
$$
The comparison principle then ensures that $\widetilde{H}(x,t)\geq\underline{U}(x,t)$ or, equivalently,
$$
\widetilde{H}(x-ct,t)\geq\underline{f}_{c,\xi_0}(x), \quad x\in\real, \quad t>0.
$$
Letting now $c'>c$, we have
$$
\widetilde{H}(x-c't,t)=\widetilde{H}(x-ct+(c-c')t,t)\geq \underline{f}_{c,\xi_0}(x+(c-c')t)\to1,
$$
as $t\to\infty$, since $c-c'<0$ and thus, for any $x$ fixed, $x+(c-c')t\to-\infty$ as $t\to\infty$. Since $c'>c>\tilde{c}$ has been chosen arbitrarily, we deduce that
\begin{equation}\label{interm15}
\lim\limits_{t\to\infty}\widetilde{H}(x-ct,t)=1, \quad c>\tilde{c},
\end{equation}
with locally uniform convergence.

Let us now pick any $c<\tilde{c}$ sufficiently close (for example, $c>kn$). In this range, $P_2$ is an unstable focus and it has been proved in \cite[Lemma 5.3 and Proposition 5.4]{IS26} that the unique trajectory contained in the stable manifold of $P_1$ comes from $P_2$ after oscillating and intersecting the $X$-axis for the last time at some point $(X_0,0)$ with $X_0>1$ before reaching $P_1$. The corresponding profiles to this trajectory have the property that $f$ oscillates infinitely many times around the constant value equal to one and decreases to zero as $\xi\to\infty$ with the behavior given in Propositions \ref{prop.P1}, \ref{prop.P1c0} or \ref{prop.P1q1}, depending on $c$ and $q$. Let $\xi_1$ be the last value of $\xi\in\real$ such that $f(\xi_1)=1$; that is, $f(\xi)\in(0,1)$ for any $\xi\in(\xi_1,\infty)$. We define the following function
$$
\overline{f}_{c,\xi_1}:=\begin{cases}
                          1, & \mbox{if } \xi\in(-\infty,\xi_1] \\
                          f(\xi), & \mbox{if } \xi>\xi_1.
                        \end{cases}
$$
It is obvious that, for any $\xi_1\in\real$, the function $\overline{U}(x,t)=\overline{f}_{c,\xi_1}(x+ct)$ is a supersolution to Eq. \eqref{eq1}. By choosing some $\xi_1>0$ (which can be chosen by a space translation), we have also the desired order at the initial time:
$$
\widetilde{H}(x,0)=1-H_0(x)\leq\overline{U}(x,0)=\overline{f}_{c,\xi_1}(x).
$$
By the comparison principle, we obtain $\widetilde{H}(x,t)\leq \overline{U}(x,t)$ or, equivalently,
$$
\widetilde{H}(x-ct,t)\leq\overline{f}_{c,\xi_1}(x), \quad x\in\real, \quad t>0.
$$
Pick now $\hat{c}<c$. It follows from the previous inequality that
$$
\widetilde{H}(x-\hat{c}t)=\widetilde{H}(x-ct+(c-\hat{c})t)\leq\overline{f}_{c,\xi_1}(x+(c-\hat{c})t)\to0,
$$
as $t\to\infty$, since $c-\hat{c}>0$ and thus, for any $x$ fixed, $x+(c-\hat{c})t\to\infty$ as $t\to\infty$. Since $\hat{c}<c<\tilde{c}$ has been chosen arbitrarily, we deduce that
\begin{equation}\label{interm16}
\lim\limits_{t\to\infty}\widetilde{H}(x-ct,t)=0, \quad c<\tilde{c},
\end{equation}
with locally uniform convergence. The convergence \eqref{convantiH1} follows from \eqref{interm15} and \eqref{interm16}. Since Lemma \ref{lem.monot} ensures that $\widetilde{H}(t)$ is decreasing for any $t>0$, the proof of the more precise convergence \eqref{convantiH2} follows by completely similar arguments as in the proof of Theorem \ref{th.Heav} and the analogous proof in \cite{KPP37}, completing the proof.
\end{proof}

We plot in Figure \ref{fig:anti} the outcome of an experiment on the evolution of the solution $\widetilde{H}(t)$ at different times, showing how it belongs to the vanishing regime.

\begin{figure}[ht!]
  \begin{center}
  \includegraphics[width=9cm,height=6cm]{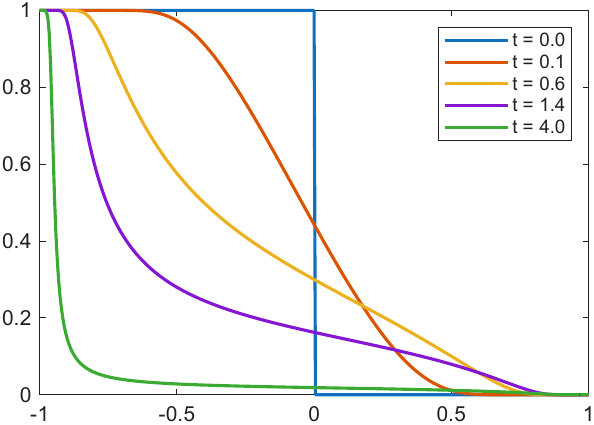}
  \end{center}
  \caption{Evolution in time of the ``anti-Heaviside" solution $\widetilde{H}(t)$. Experiment for $n=3$, $p=3$, $q=1$ and $k=0.5$}\label{fig:anti}
\end{figure}

\noindent \textbf{More general initial conditions.} As discussed in the Introduction, Theorems \ref{th.Heav} and \ref{th.antiHeav} are expressed in terms of the Heaviside-type initial conditions for the sake of simplicity. However, a direct inspection of the proofs implies that they work in the same way for the Cauchy problem associated to Eq. \eqref{eq1} with more general classes of initial conditions $u_0$, depending on the behavior of $u_0(x)$ as $x\to\pm\infty$. As an example, we state the following result.
\begin{theorem}\label{th.Heav.gen}
Let $n$, $p$, $q$, $k$ as in \eqref{range.exp} with $q>1$ and let $\overline{c}$ be the critical velocity according to \eqref{crit.vel} corresponding to these parameters. Let $\lambda_{-}(\overline{c})<0$ be the eigenvalue defined in \eqref{eigen.P2} corresponding to the velocity $\overline{c}$. Let $u_0$ be an initial condition such that $u_0(x)\in[0,1]$ for any $x\in\real$ and there exist $R_{-}<R_{+}\in\real$ and constants $C_{-}$, $C_{+}>0$ with the following properties:
\begin{equation}\label{comp.zero}
u_0(x)\leq\begin{cases}
            C_{-}e^{\overline{c}x}, & \mbox{if } \overline{c}>0,\\
            \left(\frac{|\overline{c}|}{q-1}\right)^{1/(q-1)}|x|^{-1/(q-1)}, & \mbox{if } \overline{c}<0,
          \end{cases} \quad {\rm for} \ x\in(-\infty,R_{-})
\end{equation}
and
\begin{equation}\label{comp.one}
u_0(x)\geq 1-C_{+}e^{\lambda_{-}(\overline{c})x}, \quad {\rm for} \ x\in(R_{+},\infty).
\end{equation}
Then the convergence \eqref{convH1} holds true for the solution to Eq. \eqref{eq1} with initial condition $u_0$. If, moreover, $u_0$ is nondecreasing on $\real$, then \eqref{convH2} holds true as well.
\end{theorem}
\begin{proof}[Sketch of the proof]
We infer from Proposition \ref{prop.P1} that the construction of subsolutions given in Lemma \ref{lem.sub} can be performed in the same way, provided \eqref{comp.one} holds true. Indeed, the only required thing is that the initial condition $u_0$ to lie above the critical trajectory connecting $P_2$ to $P_1$ tangent to the eigenvector $e_{-}(\overline{c})$ corresponding to the eigenvalue $\lambda_{-}(\overline{c})$ of the critical point $P_2$ when the velocity is $\overline{c}$. But this trajectory satisfies
$$
\frac{Y}{X-1}\sim\lambda_{-}(\overline{c}), \quad {\rm as} \ \xi\to\infty.
$$
By undoing the change of variable \eqref{PSchange}, we deduce that the traveling wave profiles corresponding to the critical trajectory behave like
$$
f(\xi)\sim 1-Ce^{\lambda_{-}(\overline{c})\xi}, \quad {\rm as} \ \xi\to\infty,
$$
and thus the initial comparison with subsolutions constructed as in Lemma \ref{lem.sub} is readily ensured by the property \eqref{comp.one}. In a similar way, due to the tail behavior as $\xi\to-\infty$ of the traveling wave profiles given in Proposition \ref{prop.P1}, we easily deduce that the conditions \eqref{comp.zero} ensure the comparison at the initial time with supersolutions constructed exactly as in Lemma \ref{lem.super}, and the rest of the proof is totally identical to the proof of Theorem \ref{th.Heav}.
\end{proof}
For the sake of simplicity, we refrained in the statement of Theorem \ref{th.Heav.gen} to consider the cases $\overline{c}=0$ or $q=1$, where the behavior of the traveling wave profiles as $\xi\to-\infty$ is given in Propositions \ref{prop.P1c0}, respectively \ref{prop.P1q1}. The reader can easily formulate analogous theorems as Theorem \ref{th.Heav.gen} by replacing \eqref{comp.zero} by the corresponding behaviors as $\xi\to-\infty$ established in these cases, as well as a result generalizing Theorem \ref{th.antiHeav} in the same style as Theorem \ref{th.Heav.gen}.

\medskip

\noindent \textbf{Discussion. Spreading versus vanishing for more general data $u_0$.} We complete this section with a brief and informal discussion related to the conditions on initial data $u_0$ in order to produce solutions belonging to the spreading or vanishing regime. Letting $n$, $p$, $q$ and $k$ as in \eqref{range.exp} and $u_0$ such that $0\leq u_0(x)\leq1$ for $x\in\real$, an interesting problem is to describe optimal conditions under which the solution $u$ to Eq. \eqref{eq1} with $u(x,0)=u_0(x)$, $x\in\real$, satisfies either $u(x,t)\to1$ or $u(x,t)\to0$ as $t\to\infty$, locally uniformly in $x$. The answer is strongly related to the sign of the critical velocity $\overline{c}$ defined in \eqref{crit.vel}. Indeed, if $\overline{c}<0$, \emph{any} solution with initial condition such that $u_0(x)\leq f_{\overline{c}}(x)$ for $x\in(-\infty,R)$ satisfies $u(x,t)\to0$ as $t\to\infty$, independent of its behavior as $x\to\infty$. This fact follows readily by a direct comparison argument with a suitably translated supersolution based on a traveling wave $f_{\overline{c}+\delta}$ for $\delta$ small such that $\overline{c}+\delta<0$.

Much more interesting is the case $\overline{c}>0$. In this case, the monotonicity with respect to $c$ established in \cite{IS26} entails that the unique trajectory $l_0$ starting from $P_1$ in the system \eqref{PSsyst} with $c=0$ connects directly to $P_2$ in the sense of Definition \ref{def.con}. Thus, there is a stationary solution $f_0(x)$ to Eq. \eqref{eq1} such that 
$$
0\leq f_0(x)\leq 1 \ {\rm for} \ x\in\real, \quad \lim\limits_{x\to-\infty}f_0(x)=0, \quad \lim\limits_{x\to\infty}f_0(x)=1.
$$
It is then natural to expect that the separatrix between initial conditions leading to either spreading or vanishing solutions $u(x,t)$ is the behavior of the stationary solution $f_{0}(x)$ as $x\to1$: if $u_0(x)<f_0(x)<1$ for $x\in(R,\infty)$ for some $R>0$, then one expects that $u(x,t)\to0$ as $t\to\infty$, and if $f_0(x)<u_0(x)\leq1$ for $x\in(R,\infty)$ then one expects that $u(x,t)\to1$ as $t\to\infty$. We refrain to give here rigorous results (which have to take into account also the behavior of $u_0$ as $x\to-\infty$).

Precisely, the next and final section is devoted to rigorous results on the transition between vanishing and spreading of the Heaviside solution determined by the convection coefficient $k$.

\section{Transition between vanishing and spreading regimes}\label{sec.trans}

This section is dedicated to the transition from the vanishing (that is, locally uniform convergence of $u(t)$ to zero) to the spreading (that is, locally uniform convergence of $u(t)$ to one) regime, which, as discussed in the Introduction, is one of the main effects of the convection term $k(u^n)_x$ and thus also one of the most important goals of the paper. We split this section into three parts: in a first one, we give the proof of Theorem \ref{th.transition}, while in the second one we introduce several explicit trajectories of the system \eqref{PSsyst}, whose theoretical importance is that they also provide a number of better estimates than the inequality \eqref{crit.coef} on the critical coefficient $k^*(n,p,q)$. The third and final subsection is dedicated to a self-map of the system \eqref{PSsyst} (and thus of Eq. \eqref{TWODE}) for $c=0$, which has independent interest and, at the same time, provides an alternative proof to the second part of Theorem \ref{th.transition}.

\subsection{Proof of Theorem \ref{th.transition}}

As explained above, we prove here Theorem \ref{th.transition}.
\begin{proof}[Proof of Theorem \ref{th.transition}]
As it follows from Proposition \ref{prop.P2} and \eqref{crit.vel}, we have $\overline{c}\leq kn-2\sqrt{p-q}$ and thus, if
$$
k<\frac{2\sqrt{p-q}}{n},
$$
we have $\overline{c}<0$. This proves the lower bound in \eqref{crit.coef}. The rest of the proof is divided in two parts.

\medskip

\noindent \textbf{Upper estimate in \eqref{crit.coef}.} We still have to prove the upper bound in \eqref{crit.coef}, which is more involved. To this end, let us fix $c=0$ and consider the following curve in the phase plane associated to the system \eqref{PSsyst} (with $c=0$):
\begin{equation}\label{basic.curve}
Y=X-X^n,\quad X\in[0,1],
\end{equation}
which obviously connects the critical points $P_1$ and $P_2$. Note first that the curve \eqref{basic.curve} is a trajectory of the system in the particular case $p=n$, $k=1$, $q=1$, and we note that in this case the eigenvalues of the linearization of the critical point $P_2$ are given by $\lambda_+(0)=-1$, $\lambda_{-}(0)=-(n-1)$, and the slope of the curve \eqref{basic.curve} in a neighborhood of $P_2$ is given by the vector $e_{-}(0)=(1,-(n-1))$, corresponding exactly to the eigenvalue $\lambda_{-}(0)$. Lemma \ref{lem.crit} then ensures that $\overline{c}=0$ for this particular choice of the exponents and coefficients and, in particular, $k^*(n,n,1)=1$.

Fix now $p=n$ and $q\in(1,p)$. Varying $k$, the direction of the flow of the system \eqref{PSsyst} across the curve \eqref{basic.curve} (with normal vector $\overline{n}=(1-nX^{n-1},-1)$) is given by the sign of
$$
F(X)=(k-1)nX^{n-1}(X-X^n)+X-X^q.
$$
Since $X\in[0,1]$ and $q\geq1$, if $k>1$ one can on the one hand observe that $F(X)>0$ and thus the direction of the flow coincides with the direction of the vector $\overline{n}$. We infer that the closed region limited by the curve \eqref{basic.curve} and the $X$-axis, namely
\begin{equation}\label{basic.region}
\mathcal{D}=\{(X,Y)\in\real^2: X\geq0, 0\leq Y\leq X-X^n\}
\end{equation}
is positively invariant for the system \eqref{PSsyst}. We infer from Lemma \ref{lem.flow} that the unstable trajectory $l_0$ of $P_1$ connects directly to $P_2$ in the sense of Definition \ref{def.con}. On the other hand, for $k>1$ we have
\begin{equation*}
\begin{split}
\lambda_{-}(0)&=-\frac{kn+\sqrt{(kn)^2-4(n-q)}}{2}=-\frac{kn+\sqrt{(n-2)^2+(k^2-1)n^2+4(q-1)}}{2}\\&<-(n-1),
\end{split}
\end{equation*}
whence the unique trajectory tangent to the eigenvector $e_{-}(0)$ of the critical point $P_2$ reaches $P_2$ from outside the region $\mathcal{D}$. This shows that the unstable trajectory of $P_1$ reaches $P_2$ tangent to the leading eigenvector $e_{+}(0)$ and Lemma \ref{lem.crit} implies that $\overline{c}>0$. Since $k>1$ has been chosen arbitrarily, we have thus shown that
$$
\frac{2\sqrt{n-q}}{n}\leq k^*(n,n,q)\leq1.
$$

We still have to move $p$, which up to now has been fixed as $p=n$. To this end, let us recall that we have just proved that, for $k>1$, $c=0$ and $1\leq q<p=n$, there is a direct connection (in the sense of Definition \ref{def.con}) from $P_1$ to $P_2$ entering $P_2$ tangent to the eigenvector $e_{+}(0)$. We then deduce from Lemma \ref{lem.sub} and Corollary \ref{cor.sub} that the unique trajectory entering $P_2$ tangent to the eigenvector $e_{-}(0)$ and other trajectories entering $P_2$ tangent to the eigenvector $e_{+}(0)$ cross the $Y$-axis (and produce subsolutions, as explained in Lemma \ref{lem.sub}). 

Let us now pick $p\in(1,n)$ and $1\leq q<p$. Let $Y=Y_{1,n}(X)$ be a trajectory entering $P_2$ and crossing the $Y$-axis in the system \eqref{PSsyst} with $c=0$, $q=1$, $p=n$ and $k>1$ and let $Y=Y(X)$ be a trajectory in the system \eqref{PSsyst} with $c=0$, the same value of $k>1$ and given $1\leq q<p<n$ such that there is $X_0\in[0,1]$ with $Y(X_0)=Y_{1,n}(X_0)$. The inverse function theorem ensures that the derivatives of $Y(X)$ and $Y_{1,n}(X)$ satisfy
\begin{equation}\label{interm8}
Y'(X)=-knX^{n-1}+\frac{X^q-X^p}{Y}\leq-knX^{n-1}+\frac{X-X^n}{Y}=Y_{1,n}'(X), \quad X\in[0,1],
\end{equation}
and the comparison entails that the trajectory parametrized as $Y=Y(X)$ satisfies $Y(X)\geq Y_{1,n}(X)$ for $X\in[0,X_0]$. The proof of Lemma \ref{lem.sub} implies that the trajectories of the system \eqref{PSsyst} cannot have vertical asymptotes at points $X\in[0,1]$, hence the trajectory parametrized as $Y=Y(X)$ also crosses the $Y$-axis. Since, for $1<q<p<n$ we obviously have
$$
kn+\sqrt{(kn)^2-4(n-1)}<kn+\sqrt{(kn)^2-4(p-q)},
$$
it follows that the slopes $\lambda_{-}(0)$ of the eigenvectors $e_{-}(0)$ are ordered, that is, 
$$
\lambda_{-}(0;k,n,p,q)<\lambda_{-}(0;k,n,n,1)<0.
$$ 
This fact, together with the comparison deduced from \eqref{interm8}, readily prove that the trajectory entering $P_2$ tangent to the eigenvector $e_{-}(0)$ in the system \eqref{PSsyst} for $c=0$, $k>1$ and given $1\leq q<p<n$ crosses the $Y$-axis at a point $(0,Y(0))$ with $Y(0)>0$. We then get by an application of the Poincar\'e-Bendixon's Theorem in the compact region limited by the previous trajectory $Y(X)$ and the two coordinate axis that the unstable trajectory $l_0$ of the critical point $P_1$ connects directly to $P_2$ (in the sense of Definition \ref{def.con}). In particular, this shows that $\overline{c}>0$ and, since in the previous comparison argument $k>1$ is arbitrarily chosen, we deduce that $k^*(n,p,q)\leq1$ if $1\leq q<p\leq n$.

We are left with the case $p>n$. In this case, we proceed by comparison as above, but with a change. Pick $k>p/n>1$ and let $k'=kn/p>1$. Then, in the same notation and conditions as in \eqref{interm8}, we can write
\begin{equation}\label{interm9}
Y'(X)=-knX^{n-1}+\frac{X^q-X^p}{Y}=-k'pX^{p-1}+\frac{X^q-X^p}{Y},
\end{equation}
and thus compare with the system \eqref{PSsyst} with $1\leq q<p=n$ and coefficient $k'>1$ defined above. A similar argument as in the previous paragraph shows that the unstable trajectory $l_0$ of the critical point $P_1$ directly connects to $P_2$ (in the sense of Definition \ref{def.con}) as well, and thus $k^*(n,p,q)\leq p/n$. The proof of the upper bound in \eqref{crit.coef} is now complete.

\medskip

\noindent \textbf{Equality in the lower bound if \eqref{cond.conv} holds true.} Let us assume next that $2\leq n\leq(q+1)/2$. We show that we have a trajectory directly connecting $P_1$ to $P_2$ (in the sense of Definition \ref{def.con}) in the system \eqref{PSsyst} with $c=0$, whenever $kn>2\sqrt{p-q}$. To this end, fix $c=0$ and let us consider the curve and the convection coefficient
\begin{equation}\label{complex.curve}
Y=A(X^{q+1-n}-X^{p+1-n}), \quad k=\frac{A^2(p-q)+1}{An},
\end{equation}
where $A>0$ is arbitrary. Picking the normal direction to the curve \eqref{complex.curve} as
$$
\overline{n}=\left(A(q+1-n)X^{q-n}-A(p+1-n)X^{p-n},-1\right),
$$
we obtain that the direction of the flow of the system \eqref{PSsyst} (with $c=0$ and $k$ given in \eqref{complex.curve}) across the curve \eqref{complex.curve} is given by the sign of the expression obtained as the scalar product between the vector field of the system \eqref{PSsyst} and the vector $\overline{n}$, which, after direct calculations and grouping similar terms, gives
\begin{equation*}
F(X)=A^2X^q(1-X^{p-q})G(X),
\end{equation*}
with
$$
G(X):=p-q+(q+1-n)X^{q+1-2n}-(p+1-n)X^{p+1-2n}.
$$
Note that $G(0)=p-q>0$ by \eqref{range.exp}, while $G(1)=0$. Moreover,
$$
G'(X)=X^{q-2n}\left[(q+1-n)(q+1-2n)-(p+1-n)(p+1-2n)X^{p-q}\right]=X^{q-2n}H(X).
$$
It is exactly at this point where the condition \eqref{cond.conv} enters decisively into play. Indeed, since $q+1-2n\geq0$, we have that $H(0)\geq0$, while $p>q$ also implies that $p+1-n>p+1-2n>0$ and thus the function $H(X)$ is decreasing with $X$ for $X\in[0,1]$. Noting that $H(1)<0$, we infer that $H$ has a unique zero $X_0\in[0,1)$ (with $X_0=0$ exactly when $q+1=2n$). Since $G'$ and $H$ have the same sign, it follows that $X_0$ is a local maximum point for $G$ and then $G$ decreases on $(X_0,1)$. Moreover, $G(X_0)\geq G(0)=p-q>0$, and thus it is obvious that $G(X)\geq0$ for any $X\in[0,1]$ (with $X=1$ as the single zero for $G$ in $[0,1]$). This sign of $G$ (which is the same as the sign of $F$ for $X\in[0,1]$) entails that the region $\mathcal{D}$ limited by the curve \eqref{complex.curve} and the $X$-axis is positively invariant. Lemma \ref{lem.flow} (with \eqref{complex.curve} in the role of the curve $\Gamma$) ensures that the unstable trajectory from $P_1$ connects directly to $P_2$. Since $A>0$ has been chosen arbitrarily, the above analysis holds true for any $k$ as defined in \eqref{complex.curve}. Optimizing with respect to $A$ and taking into account that
$$
\inf\limits_{A>0}\frac{A^2(p-q)+1}{An}=\frac{2\sqrt{p-q}}{n},
$$
we readily find that $k^*(n,p,q)=2\sqrt{p-q}/n$, achieving thus the equality in the lower bound in \eqref{crit.coef}.
\end{proof}

\subsection{Explicit trajectories and further estimates for $k^*(n,p,q)$}\label{subsec.exp}

In this section, we give a list of explicit trajectories of the system \eqref{PSsyst} (and, equivalently, explicit traveling wave solutions to Eq. \eqref{eq1} whenever it is possible to integrate the differential equation given by the trajectory). Besides the role of examples, these explicit trajectories will play the theoretical role of either establishing an exact value for the critical (transition) coefficient $k^*(n,p,q)$ under suitable conditions between $n$, $p$ and $q$, or improving the estimates in \eqref{crit.coef}. We mention that a few of the calculations leading to the explicit trajectories below have been performed with the help of a symbolic calculation software.

$\bullet$ We first give a general explicit solution in the particular case $p=n$, $q=1$, generalizing the solution \eqref{basic.curve} employed in the proof of Theorem \ref{th.transition}. Let us stress here that this particular solution allows us to find the exact value of $\overline{c}$ when $p=n$ and $q=1$. More precisely, consider the following trajectory and velocity:
\begin{equation}\label{curve0}
Y=k(X-X^n), \quad c=\frac{k^2-1}{k}.
\end{equation}
Direct calculation of the flow of the system \eqref{PSsyst} with the previous value of $c$ show that \eqref{curve0} is a trajectory. After integration it leads to the traveling wave profiles
$$
f(\xi)=\left(1+Ce^{-k(n-1)\xi}\right)^{-1/(n-1)}, \quad C>0.
$$
Coming back to the trajectory \eqref{curve0}, we deduce by direct calculation that, for the value of $c$ given in \eqref{curve0}, the two eigenvalues of the critical point $P_2$ are given by 
$$
\lambda_{-}(c)=-k(n-1), \quad \lambda_+(c)=-\frac{1}{k}, \quad |\lambda_{-}(c)-\lambda_{+}(c)|=\frac{k^2(n-1)-1}{k}.
$$
while the slope of the trajectory \eqref{curve0} as $X\to1$ is 
$$
\frac{d}{dX}[k(X-X^n)]\Big|_{X=1}=-k(n-1)=\lambda_{-}(c).
$$
It follows from the previous equalities that the trajectory and velocity in \eqref{curve0} are critical (that is, $\overline{c}=(k^2-1)/k$) provided that $k^2(n-1)-1>0$, or equivalently $k^2>1/(n-1)$.

$\bullet$ Letting $p=2n-1$, $q=n$ and $k>\frac{2}{\sqrt{n}}$, we have the following explicit trajectories: 
\begin{equation}\label{expl2}
Y=c_i(X-X^n), \quad c_1=\frac{kn+\sqrt{k^2n^2-4n}}{2n}, \quad c_2=\frac{kn-\sqrt{k^2n^2-4n}}{2n},
\end{equation}
noting that $c_i>0$ for $i=1,2$. Similarly as in the previous case, these trajectories correspond to the explicit traveling wave profiles 
$$
f(\xi)=\left(1+Ce^{-c_i(n-1)\xi}\right)^{-1/(n-1)}, \quad C>0.
$$
Note that the slope of the trajectories \eqref{expl2} as $X\to1$ is
$$
\frac{d}{dX}[c_i(X-X^n)]\Big|_{X=1}=-c_i(n-1), \quad i=1,2,
$$
while by replacing, for the easiness of the calculations,
$$
k=\frac{c^2n+1}{cn}, \quad c\in\{c_1,c_2\},
$$
in the expressions of the eigenvalues of the critical point $P_2$ (according to \eqref{eigen.P2}), we find that the eigenvalues are equal to $-1/c_i$, respectively $-c_i(n-1)$. Thus, one of the two velocities $c_1$ or $c_2$ is critical whenever $c_i>1/\sqrt{n-1}$. Since $k>2/\sqrt{n}$, it is easy to show that $c_2<1/\sqrt{n-1}$ for any such $k$, hence the only critical velocity can be $c_1$. Equating $c_1>1/\sqrt{n-1}$, we conclude that
$$
\overline{c}=c_1=\frac{kn+\sqrt{k^2n^2-4n}}{2n} \quad {\rm for \ any} \quad k>\frac{2n-1}{n\sqrt{n-1}}.
$$

$\bullet$ Letting $c=0$, $p=2n-1$, $q=1$ and $k=(n+1)/n$, we have the following explicit trajectory:
\begin{equation}\label{curve1}
Y=X-X^n,
\end{equation}
which, after integration, leads to the traveling wave profiles
$$
f(\xi)=\left(1+Ce^{-(n-1)\xi}\right)^{-1/(n-1)}, \quad C>0.
$$
Moreover, with this choice of exponents, $\lambda_{-}(0)=1-n$ and then $e_{-}(0)=(1,1-n)$ coincides with the slope of the curve \eqref{curve1} at $X=1$, hence the trajectory \eqref{curve1} is the critical one and $\overline{c}=0$. We deduce that
$$
k^*(n,2n-1,1)=\frac{n+1}{n}
$$
and, by the same argument of comparison and monotonicity with respect to $p$ as in the first part of the proof of Theorem \ref{th.transition}, we further derive that
\begin{equation}\label{estimate1}
k^*(n,p,1)<\frac{n+1}{n}, \quad {\rm for \ any} \ p\in(1,2n-1).
\end{equation}
Note that \eqref{estimate1} improves the upper bound $p/n$ established in \eqref{crit.coef} for $p\in(n+1,2n-1)$.

$\bullet$ In this part, we prove the following result.
\begin{proposition}
Let $q\in[1,n-1]$ and $p=n$. Then,
$$
k^*(n,n,q)>\frac{2\sqrt{n-q}}{n},
$$
that is, the lower bound in \eqref{crit.coef} is strict for this choice of exponents.
\end{proposition}
\begin{proof}
For $q=1$, we have proved that $k^*(n,n,1)=1>2\sqrt{n-1}/n$ in the proof of Theorem \ref{th.transition}. Let then
$$
p=n\geq2, \quad q\in(1,n-1], \quad k=\frac{2\sqrt{n-q}}{n}, \quad c=0,
$$
and consider the curve connecting $P_1$ to $P_2$ in the phase plane associated to the system \eqref{PSsyst} (with $c=0$)
\begin{equation}\label{curve.eps}
Y=k(1+\epsilon)(X^q-X^n),
\end{equation}
with $\epsilon>0$ to be chosen later. The direction of the flow of the system \eqref{PSsyst} across the curve \eqref{curve.eps} (with normal direction $\overline{n}=(k(1+\epsilon)(qX^{q-1}-nX^{n-1}),-1)$) is given by the sign of the expression
$$
\mathcal{F}_{\epsilon}(X)=\frac{4(n-q)}{n^2}(X^q-X^n)\left[(1+\epsilon)^2qX^{q-1}-\epsilon(1+\epsilon)nX^{n-1}-\frac{n^2}{4(n-q)}\right],
$$
for $X\in[0,1]$. Denoting by
$$
F_{\epsilon}(X):=(1+\epsilon)^2qX^{q-1}-\epsilon(1+\epsilon)nX^{n-1}-\frac{n^2}{4(n-q)}, \quad X\in[0,1],
$$
we notice that, on the one hand, after easy algebraic manipulations,
$$
F_{\epsilon}(1)=-\frac{1}{4(n-q)}\left[n-2q-2\epsilon(n-q)\right]^2\leq0.
$$
On the other hand, we calculate
\begin{equation}\label{interm10}
F_{\epsilon}'(X)=(1+\epsilon)X^{q-2}\left[q(q-1)(1+\epsilon)-n(n-1)\epsilon X^{n-q}\right].
\end{equation}
Choose now
\begin{equation}\label{interm11}
\epsilon:=\frac{q(q-1)}{n(n-1)-q(q-1)},
\end{equation}
a choice which ensures that $F_{\epsilon}'(X)\geq0$ for any $X\in[0,1]$, as it follows readily from \eqref{interm10}. Thus, $F_{\epsilon}(X)\leq F_{\epsilon}(1)\leq0$ for any $X\in[0,1]$. We then infer from Lemma \ref{lem.flow} (with \eqref{curve.eps} in place of $\Gamma$) that the unstable trajectory $l_0$ goes out of the critical point $P_1$ outside the closed region $\mathcal{D}$ limited by the curve \eqref{curve.eps} and the $X$-axis. We next compare the slope of the curve \eqref{curve.eps} at $X=1$ with the slope of the eigenvector $e_{-}(0)$ corresponding to the eigenvalue $\lambda_{-}(0)$. Indeed, by the choice of $k$, we have that $\lambda_{-}(0)=-kn/2$, while
$$
Y'(X)\Big|_{X=1}=k(1+\epsilon)(q-n).
$$
Taking into account the choice of $\epsilon$ in \eqref{interm11}, we deduce that
\begin{equation}\label{interm12}
\frac{1+2\epsilon}{2+2\epsilon}n-q=\frac{(n-q)(n-q-1)}{2(n-1)}\geq0,
\end{equation}
recalling that $q\leq n-1$. The inequality \eqref{interm12} implies that $2(1+\epsilon)(n-q)>n$ and thus
$$
\lambda_{-}(0)=-\frac{kn}{2}>-k(1+\epsilon)(n-q)=Y'(X)\Big|_{X=1},
$$
which ensures that the eigenvector $e_{-}(0)=(1,\lambda_{-}(0))$ and the critical trajectory entering $P_2$ tangent to it come from the interior of the closed region $\mathcal{D}$ limited by the curve \eqref{curve.eps} and the $X$-axis. Due to the direction of the flow on the curve \eqref{curve.eps}, we have just proved that the unstable trajectory $l_0$ stemming from $P_1$ cannot connect to $P_2$ and has to intersect the $X$-axis at a point $(X_0,0)$ with $X_0>1$. This fact, together with Lemma \ref{lem.super}, complete the proof.
\end{proof}

\medskip

It is now the right time to move on towards more involved and more general explicit trajectories. To this end, let us consider in the next examples general exponents $n$ and $q$ such that $n>(q+1)/2$ (as the opposite range is already covered in Theorem \ref{th.transition}) and velocity $c=0$. The next result establishes the exact value of the transition coefficient $k^*(n,p,q)$ in two rather general cases.
\begin{proposition}
Let $n$ and $q$ be as in \eqref{range.exp} and such that $n>(q+1)/2$. Then we have
\begin{equation}\label{estimate3}
k^*(n,2n-1,q)=\frac{2n+q+1}{n\sqrt{2(q+1)}}, \quad {\rm provided} \ n>\frac{3(q+1)}{2}
\end{equation}
and
\begin{equation}\label{estimate4}
k^*\left(n,n+\frac{q-1}{2},q\right)=\frac{2}{\sqrt{2(q+1)}}, \quad {\rm provided} \ n>q+1.
\end{equation}
\end{proposition}
\begin{proof}
Let
$$
p=2n-1, \quad k=\frac{2n+q+1}{n\sqrt{2(q+1)}}.
$$
With these choices, the following curve
\begin{equation}\label{curve3}
Y=\frac{2}{\sqrt{2(q+1)}}\left(X^{(q+1)/2}-X^n\right)
\end{equation}
is an explicit trajectory of the system \eqref{PSsyst} with $c=0$. Note first that
\begin{equation*}
\begin{split}
kn-2\sqrt{p-q}&=\frac{2n+1+q}{\sqrt{2(q+1)}}-2\sqrt{2n-1-q}\\&=\frac{1}{\sqrt{2(q+1)}}(\sqrt{2n-1-q}-\sqrt{2(q+1)})^2\geq0,
\end{split}
\end{equation*}
with equality if and only if $n=3(q+1)/2$. We next calculate the eigenvalue $\lambda_{-}(0)$ and the slope of the trajectory \eqref{curve3} as $X\to1$. We have
$$
Y'(X)\Big|_{X=1}=-\frac{2n-1-q}{\sqrt{2(q+1)}}=\begin{cases}
                                                 \lambda_{-}(0), & \mbox{if } n>\frac{3(q+1)}{2}, \\[1mm]
                                                 \lambda_{+}(0), & \mbox{if } n<\frac{3(q+1)}{2},
                                               \end{cases}
$$
which entails that the trajectory \eqref{curve3} reaches the critical point $P_2$ tangent to the eigenvector $e_{-}(0)$ and is thus critical only if $n>3(q+1)/2$. This completes the proof of \eqref{estimate3}.

\medskip

Let next
$$
p=n+\frac{q-1}{2}, \quad k=\frac{2}{\sqrt{2(q+1)}}.
$$
With these choices, the same curve \eqref{curve3} as in the previous item is an explicit trajectory of the system \eqref{PSsyst} with $c=0$. Note first that
$$
kn-2\sqrt{p-q}=\frac{2n}{\sqrt{2(q+1)}}-2\sqrt{n-\frac{q+1}{2}}=\frac{1}{\sqrt{2(q+1)}}(\sqrt{2n-1-q}-\sqrt{q+1})^2\geq0,
$$
with equality if and only if $n=q+1$. We next calculate the eigenvalue $\lambda_{-}(0)$ and the slope of the trajectory \eqref{curve3} as $X\to1$. We have
$$
Y'(X)\Big|_{X=1}=-\frac{2n-1-q}{\sqrt{2(q+1)}}=\begin{cases}
                                                 \lambda_{-}(0), & \mbox{if } n>q+1, \\[1mm]
                                                 \lambda_{+}(0), & \mbox{if } n<q+1,
                                               \end{cases}
$$
which entails that the trajectory \eqref{curve3} reaches the critical point $P_2$ tangent to the eigenvector $e_{-}(0)$ and is thus critical only if $n>q+1$. The proof of \eqref{estimate4} is now complete.
\end{proof}

\subsection{A self-map for $c=0$}

The last result of this paper is the establishment of a self-map of the system \eqref{PSsyst}; that is, a transformation between different sets of parameters $(n,p,q,k)$ of the system \eqref{PSsyst}. The idea of an existence of such a self-map comes from the fact that, as noticed in the previous sections, a number of explicit trajectories of the system appear twice (for example, the curve \eqref{curve3}, or even more obvious, the simpler curve \eqref{curve1} which is a trajectory either for $p=n$, $q=1$, $c=0$, $k=1$ or for $p=2n-1$, $q=1$, $c=0$, $k=(n+1)/n$). Let us thus fix $c=0$ and start from a different system employed in \cite[Section 9]{IS26}, namely
\begin{equation}\label{PSsyst2}
\left\{\begin{array}{ll}U'=(p-q)UV, \\V'=-knU^{(2n-1-q)/(p-q)}V-U+1-\frac{q+1}{2}V^2,\end{array}\right.
\end{equation}
where the new dependent variables $(U,V)$ and the new independent variable $\eta$ are defined as
\begin{equation}\label{change2}
U=X^{p-q}, \quad V=YX^{-(1+q)/2}, \quad \eta'(\xi)=X(\xi)^{(q-1)/q}.
\end{equation}
This system has been employed precisely in \cite[Section 9]{IS26} for the proof of Proposition \ref{prop.P1c0} in the case $n>(q+1)/2$. However, the previous change of variable and the system \eqref{PSsyst2} can be used in any range of exponents (with the only problem that, in some ranges, its vector field might not be of class $C^1$ and some theoretical results in dynamical systems cannot apply). We further introduce the rescaling
\begin{equation}\label{resc}
V=dW, \quad \eta=d\eta_1, \quad d=\sqrt{\frac{2}{q+1}}
\end{equation}
and, in variables $(U,W)$, the system \eqref{PSsyst2} is transformed into
\begin{equation}\label{PSsyst3}
\left\{\begin{array}{ll}U'=\frac{2(p-q)}{q+1}UW, \\W'=-\frac{kn\sqrt{2}}{\sqrt{1+q}}U^{(2n-1-q)/(p-q)}W-U+1-W^2,\end{array}\right.
\end{equation}
where derivatives are taken with respect to the new independent variable $\eta_1$ introduced in \eqref{resc}. In order to derive a self-map, we equate the coefficients and exponents of the system \eqref{PSsyst3} for two different sets of parameters $(n_1,p_1,q_1,k_1)$, respectively $(n_2,p_2,q_2,k_2)$; that is,
\begin{equation}\label{interm13}
\frac{2(p_1-q_1)}{q_1+1}=\frac{2(p_2-q_2)}{q_2+1}, \quad \frac{k_1n_1}{\sqrt{q_1+1}}=\frac{k_2n_2}{\sqrt{q_2+1}}, \quad
\frac{2n_1-q_1-1}{p_1-q_1}=\frac{2n_2-q_2-1}{p_2-q_2}.
\end{equation}
Solving the system \eqref{interm13}, we readily obtain the following correspondence of the parameters:
\begin{equation}\label{self-map}
p_2=(p_1+1)\frac{n_2}{n_1}-1, \quad q_2=(q_1+1)\frac{n_2}{n_1}-1, \quad k_2=k_1\sqrt{\frac{n_1}{n_2}}.
\end{equation}
By undoing the changes of variables \eqref{resc} and \eqref{change2}, we have just established that, for $c=0$, the systems \eqref{PSsyst} for the sets of parameters $(n_1,p_1,q_1,k_1)$ and $(n_2,p_2,q_2,k_2)$ related by the equations \eqref{self-map} are topologically equivalent. Note also that, if we fix $(n_1,p_1,q_1,k_1)$, then $p_2$, $q_2$ and $k_2$ are uniquely determined by the quotient $n_2/n_1$, thus $n_2$ can be chosen arbitrarily in \eqref{self-map}.

We believe that this self-map has independent interest and some more applications of it (in form of explicit curves for some ranges or better estimates for $k^*(n,p,q)$) can be found. We will use it here to give an \emph{alternative proof for the second part of Theorem \ref{th.transition}}. Pick $n_1=1$ and $q_1\geq1$. We then obtain from \eqref{self-map} that $q_2=n_2(1+q_1)-1\geq 2n_2-1$, or, equivalently, $n_2\leq(q_2+1)/2$. Conversely, if we pick $n_2$ and $q_2$ such that \eqref{cond.conv} holds true, then we can fix $n_1=1$ and $q_1=(1+q_2)/n_2-1\geq1$ such that \eqref{self-map} holds true. Thus, any system \eqref{PSsyst} with $c=0$ and $n$, $q$ satisfying \eqref{cond.conv} is topologically equivalent by the previous self-map to a sysrem \eqref{PSsyst} with $c=0$ and $n=1$. But for $n=1$ it is well-known (see \cite[Remark, p.5-6]{IS26} and \cite{IS25}) that there exists a direct connection (in the sense of Definition \ref{def.con}) from $P_1$ to $P_2$ in the phase plane, provided $k>2\sqrt{p-q}$. Coming back to the notation corresponding to the self-map, pick $k_2$ such that
\begin{equation}\label{interm14}
k_2n_2>2\sqrt{p_2-q_2}.
\end{equation}
By applying \eqref{self-map}, the latter inequality is transformed (in variables $(n_1,p_1,q_1,k_1)$ with $n_1=1$) into $k_1\sqrt{n_2}>2\sqrt{n_2(p_1-q_1)}$, or equivalently $k_1>2\sqrt{p_1-q_1}$. But in this case, we know that there exists a direct connection from $P_1$ to $P_2$, and the topological equivalence given by the self-map implies that such a direct connection also exists for the system \eqref{PSsyst} with $c=0$ and parameters $(n_2,p_2,q_2,k_2)$. Since $k_2$ has been chosen arbitrarily to satisfy \eqref{interm14}, we infer that $k^*(n_2,p_2,q_2)=2\sqrt{p_2-q_2}/n_2$, whenever \eqref{cond.conv} is fulfilled. The proof is completed by dropping the subindex from the notation.

\bigskip

\noindent \textbf{Acknowledgements} This work is partially supported by the Spanish project PID2024-160967NB-I00, funded by the Spanish Agency for Research (Ministry of Science, Innovation and Universities of Spain) and FEDER.

\bigskip

\noindent \textbf{Data availability} Our manuscript has no associated data.

\bigskip

\noindent \textbf{Conflict of interest} The authors declare that there is no conflict of interest.

\bibliographystyle{plain}

\end{document}